\documentclass[12pt]{amsart}
\usepackage{amsmath,amsthm,amsfonts,amscd,amssymb,eucal,latexsym,mathrsfs}
\usepackage[all]{xy}

\newtheorem{theorem}{Theorem}[section]
\newtheorem{corollary}[theorem]{Corollary}
\newtheorem{lemma}[theorem]{Lemma}
\newtheorem{proposition}[theorem]{Proposition}

\theoremstyle{definition}
\newtheorem{definition}[theorem]{Definition}
\newtheorem{remark}[theorem]{Remark}

\newtheorem{example}[theorem]{Example}

\theoremstyle{plain}

\newtheorem{claim}{Claim}
\theoremstyle{definition}

\theoremstyle{remark}

\newcommand{\Aut}{{\rm Aut}}

\newcommand{\Zb}{{\mathbb Z}}

\newcommand{\Nb}{{\mathbb N}}

\newcommand{\IN}{{\mathbb N}}
\newcommand{\IZ}{{\mathbb Z}}

\newcommand{\os}{{\boldsymbol{s}}}

\newcommand{\sP}{{\mathscr P}}

\newcommand{\NSA}{{\rm NSA}}

\newcommand{\SA}{{\rm SA}}

\newcommand{\bs}{\mathbf{\Sigma}}
\newcommand{\bb}[1]{\llbracket {#1}\rrbracket}

\newcommand{\e}{\varepsilon}
\newcommand{\R}{R}
\newcommand{\F}{\mathbf{F}}
\newcommand{\G}{\Gamma}
\DeclareMathOperator{\cost}{\mathrm{cost}}
\DeclareMathOperator{\dom}{\mathrm{dom}}
\DeclareMathOperator{\ran}{\mathrm{ran}}
\DeclareMathOperator{\tr}{{\mathrm tr}}
\newcommand{\p}{\varphi}
\newcommand{\IP}{\mathbb{P}}

\newcommand{\Res}{{\mathrm{Res}}}

\newcommand{\sym}{{\pm}}
\newcommand{\slower}{\underline{s}}
\newcommand{\acts}{\curvearrowright}

\allowdisplaybreaks

\date{\today}
\usepackage{stmaryrd}

\title[Orbit equivalence and sofic approximation]{Orbit equivalence and sofic approximation}

\author{K. Dykema}
\author{D. Kerr}
\author{M. Pichot}

\begin{document}
\maketitle

\begin{abstract}
Given an ergodic probability measure preserving dynamical system $\G\acts (X,\mu)$, where $\G$ is a finitely generated countable group, we show that the asymptotic growth of the number of finite models for the dynamics, in the sense of sofic approximation, is an  invariant of orbit equivalence. We then prove an additivity formula for free products of orbit structures with amenable (possibly trivial) amalgamation. In particular, we obtain purely combinatorial proofs of several results in orbit equivalence theory.
\end{abstract}

\section{Introduction}

In this paper we introduce a new invariant of probability measure preserving dynamical systems using the idea of sofic approximation. This invariant is  combinatorial in nature. It is defined by counting the number of sofic models on finite sets to within a given precision, and taking a limiting value of the asymptotic growth of these numbers as the finite sets get larger and larger and the precision gets better and better (see  Section \ref{S- definition} for a detailed discussion).

Let $\G\acts(X,\mu)$ be a probability measure preserving (pmp) action of a finitely generated group $\G$ on a standard probability space $(X,\mu)$, and let $F$ be a finite dynamical generating set for $\G\acts(X,\mu)$, in the sense of Definition \ref{D-gen}. For instance, if the action $\G\acts (X,\mu)$ has a finite generating partition $\sP$, then $F$ can be taken to be the union of the partition $\sP$ and a finite generating set of $\G$, viewed as a set of measure-preserving (partial) isomorphisms of  $(X,\mu)$.
 We associate to  $F$ a value $s(F)$  in $\{-\infty\}\cup [0,\infty[$, called the sofic dimension of $F$, which satisfies the following invariance property.

\begin{theorem}\label{th1} The sofic dimension $s(F)$ of $F$ is an invariant of orbit equivalence. Namely, it depends only on the orbit partition of the action $\G\acts(X,\mu)$.
\end{theorem} 

 We recall that the \emph{orbit partition} of an action $\G\acts (X,\mu)$ is the measure equivalence relation $R$ on $(X,\mu)$ whose classes are the orbits of the action $\G\acts (X,\mu)$. Two actions are \emph{orbit equivalent} to each other if they have isomorphic orbit partition (see \cite{Kechris2004} for a recent survey on orbit equivalence). The notion of finite dynamical generation considered in  
Definition \ref{D-gen} is more general than the usual dynamical notion requiring the existence of a finite generating partition for $\G\acts (X,\mu)$. In fact, it is precisely the orbit equivalence generalization of this notion, in that it only requires $\G\acts (X,\mu)$ to be orbit equivalent to an action of a finitely generated group having such a partition. We observe in Section \ref{S- definition} that this holds for instance if $\G\acts(X,\mu)$ is ergodic. On the other hand, it is more restrictive than the usual notion of generating sets for equivalence relations (also called graphings, see Definition \ref{D-gen old}), in that it requires generation of the measure algebra of $(X,\mu)$ in addition to generation of the classes of the relation $R$.   It is natural to reformulate Theorem \ref{th1} in terms of abstract pmp equivalence relations which admit a finite dynamical generating set, and this is done in Section \ref{S-OE section}, see Theorem \ref{P-generating}. We shall denote by $s(R)$ the common value of $s(F)$ taken over all finite dynamical generating sets $F$.
  
Under a mild technical assumption called ``$s$-regularity" we then show (see Theorem \ref{T- amalgamated}) the following additivity formula for amalgamated free products (see  \cite{Gaboriau2000a} D\'ef. IV.6 for the definition):

\begin{theorem} \label{th2} Assume that the pmp equivalence relation $R$ is an amalgamated free product of the form $R=R_1*_{R_3}R_2$, where the finitely generated relations $R_1$ and $R_2$ are $s$-regular and $R_3$ is amenable. Then $R$ is $s$-regular and
\[
s(R)=s(R_1)+s(R_2)-1+\mu(D).
\]
where $D$ is a fundamental domain of the finite component of $R_3$.
\end{theorem}

\begin{remark}
Given a nonprincipal ultrafilter $\omega\in \beta(\IN)\setminus \IN$, we can modify the definition of $s$ to obtain another invariant $s_\omega(R)\leq s(R)$. If $R$ is $s$-regular, then $s_\omega(R)= s(R)$. With this variation one recovers the formula $s_\omega(R)=s_\omega(R_1)+s_\omega(R_2)-1+\mu(D)$ for amalgamated products of the form $R=R_1*_{R_3}R_2$, where $R_1$ and $R_2$ have a finite dynamical generating set (but are not necessarily $s$-regular) and $R_3$ is amenable.  
\end{remark}

These results provide a new approach to orbit equivalence theory (for pmp actions) that relies essentially on counting arguments. For example, they imply that two free groups $\F_p$ and $\F_q$ with $p\neq q$ have no orbit equivalent free ergodic pmp action \cite{Gaboriau2000a}, or that every treeable ergodic pmp  equivalence relation (in the sense of \cite[D\'efinition I.2] {Gaboriau2000a})  is sofic   \cite{Elek2010d} (see also \cite{Pau}) and has the expected cost (as computed in \cite{Gaboriau2000a}). Every treeable ergodic pmp  equivalence relation is $s$-regular (see Corollary \ref{C-treeable}).

Let us now give some more details and historical background on free entropy, orbit equivalence, and sofic approximations.

Given a finite von Neumann algebra $M$,  Voiculescu introduced several quantities, in particular \emph{free entropy} and \emph{free entropy dimension}, defined by taking the asymptotic growth rate of the volume  of certain matricial microstates associated to a finite set $F$ of self-adjoint elements in  $M$ (see in particular \cite{Voiculescu1993} and \cite{Voiculescu1998}); for an equivalent packing dimension approach to free entropy dimension, see \cite{Jung2002}. Our invariant is a combinatorial analogue of these quantities for dynamical systems. Voiculescu's free entropy theory allowed him to  solve several longstanding open problems on finite von Neumann algebras. The computation of his invariants, that we shall simply denote  $\delta(F)$ here,  has by now been done for many finite sets $F\subset M$. We refer to the introduction of \cite{Brown2006e} for a recent overview of results relevant to the present paper.
The analogue of Theorem \ref{th1} is not known for arbitrary finite generating sets of a finite von Neumann algebra $M$, although it is known that $\delta(F)\leq 1$ for all finite generating sets $F\subset M$ when $M$ is the hyperfinite $\mathrm{II}_1$ factor, or (much) more generally, when $M$ is strongly 1-bounded \cite{Jung2005}. Despite much recent progress, the question of distinguishing the $\mathrm{II}_1$ factors of two free groups $\F_p$ and $\F_q$ with $p\neq q$ up to isomorphism remains open.

The \emph{cost} of a pmp equivalence relation, denoted $\cost(R)$, was studied by Gaboriau in his breakthrough paper \cite{Gaboriau2000a}. One of the main results of \cite{Gaboriau2000a} is the additivity formula for the cost,  
\[
\cost(R)=\cost(R_1)+\cost(R_2)-1+\mu(D),
\]
for a free product $R=R_1*_{R_3}R_2$ of finitely generated equivalence relations over an amenable subrelation $R_3$. In the context of Voiculescu's free entropy, a similar formula for certain finite sets of free products of the form $M=M_1*_{M_3}M_2$ amalgamated over an amenable von Neumann algebra $M_3$ is established in \cite{Brown2006e}. Furthermore,
 a relative version of Voiculescu's free entropy theory was shown by Shlyakhtenko in \cite{Shlyakhtenko2003} to provide another orbit equivalence invariant $\delta(R)$. Shlyakhtenko proves that  
 \[
 \delta(R)=\delta(R_1)+\delta(R_2)
 \] 
 whenever $R=R_1*R_2$ is a free product of finitely generated pmp equivalence relations.

The relation between these invariants is unclear in general. 
In all known cases, we have $\slower(R)=s_\omega(R)=s(R)=\delta(R)=\cost(R)$, 
but it seems unlikely that these equalities hold in general. They do hold if $R$ is treeable, as a consequence of the fact that these invariants take the same value for amenable equivalence relations, and behave similarly under free products.  For example, if $R$ is the orbit equivalence relation of a free pmp action of a  free group $\F_p$ on $p$ generators, then $s(R)=\delta(R)=\cost(R)=p$. Note that, in particular, these invariants cannot distinguish between any two different actions of $\F_p$ up to orbit equivalence (there are uncountably many such actions \cite{Gaboriau2005e}), but it is not known whether this remains the case for more general acting groups.
 We note that the equality $s(R)= \cost(R)$ implies in particular that $R$ is sofic.

Sofic groups were introduced by Gromov in \cite[4.G]{Gromov1999} (see also \cite{Weiss2000}) and have generated a wealth of activity in recent years. While it seems difficult to construct groups that are not sofic, many of the conjectures that are formulated for all countable groups are known to be true for sofic groups (see e.g. \cite{Gromov1999, Elek2005c}).
The invariant $s$ measures ``how sofic"  the system under consideration is. Understandably, it is maximal for treeable equivalence relations, that is, in the freest case (a similar invariant for groups would give $s(\F_p)=p$, which is the maximal value for a $p$-generated group). The notion of sofic  equivalence relations was introduced and studied in \cite{Elek2010d} (the definition we are using in the present paper is taken
from \cite{Oz} and was studied in \cite{Pau}). Sofic approximations have already found several applications to dynamical systems, in particular through  Bowen's recent construction of measure-conjugacy entropy invariants for actions of sofic groups \cite{Bowen2010}. Besides the analogy with Voiculescu's free entropy dimension mentioned already, it is interesting to note that the proof of Theorem \ref{th1} also has some similarity with the proof of invariance of entropy under change of finite generating partition \cite[Theorem 2.2]{Bowen2010}, see also the proof of Theorem 2.6 in \cite{Kerr2010a}. 
  In an other direction,
 we mention that, by considering actions by Bernoulli shifts, our results can also be applied to finitely generated groups (although this is not the shortest way to obtain purely group-theoretic results). For example, they imply  that  an amalgamated free product of sofic groups over an amenable group is sofic, a recent result obtained independently by \cite{Elek2010c} and \cite{Pau} (see also \cite{Elek2003} for the case of free products with no amalgamation and \cite{Collins2010} for the case of amalgamation over monotileably amenable subgroups). We will study these aspects from the viewpoint of the sofic dimension in a separate paper \cite{dkp}.
 
While working on this project we learned that Miklo\'s Ab\'ert, Lewis Bowen, and Nikolai Nikolov also defined and studied the same notion of sofic dimension for groups, and it is in fact their terminology that we have adopted. Our paper answers a question of Miklo\'s Ab\'ert, who asked whether the theory can be extended to measure-preserving group actions  (see Item 14 in Section 4 of \cite{Abert2009}).

\medskip

{\bf Acknowledgments.} The first author was partially supported by NSF grant DMS-0901220.
 Some of this research was conducted while he was attending the
 Erwin Schr\"odinger Institute in Vienna and he would like to thank
 the institute and the organizers of the program on Bialgebras and
 Free Probability. The second author was partially supported by NSF grant DMS-0900938.
He would like to thank Yasuyuki Kawahigashi for hosting his
January 2010 visit to the University of Tokyo during which
the initial stages of this work were carried out. The third author was partially supported by JSPS and thanks Narutaka Ozawa for helpful discussions on the subject. 

\section{Definition of $s(F)$, dynamical generating sets, and $s$-regularity} \label{S- definition}

Throughout $(X,\mu )$ denotes a standard probability space. We refer to \cite{Kechris2004} for a  recent reference on measured equivalence relations.
 
Let $\R$ be a measured  equivalence relation on $(X,\mu )$.  
We write $[\R]$ for the full group of $\R$ and $\llbracket R\rrbracket$ for the set of all partial measure-preserving transformations of $X$ whose graph is included in $R$, where we identify two transformations if they coincide almost surely in the usual way.  The product $st$ of two transformations $s,t\in \llbracket R\rrbracket$ is defined to be  $(st)x:=s(t(x))$ for all $x$ in  $\dom st=t^{-1}(\ran t\cap \dom s)$,  where $\dom s$ and $\ran s$ denote  respectively the domain and the range of $s\in \llbracket R\rrbracket$. We write $1\in [R]$ for the identity transformation of $X$, and $0\in \llbracket R\rrbracket$ for the negligible transformation of $X$ (with empty domain). The set $\llbracket R\rrbracket$  is an inverse semigroup under composition and inverse. 
Note that we have a partial additive structure on $\llbracket R\rrbracket$, with neutral element $0\in \llbracket R\rrbracket$, coming from the additive structure on the von Neumann algebra $LR$. Namely, if $s_1 , \dots ,s_k$ are elements of $\llbracket R\rrbracket$ with pairwise disjoint domains and pairwise disjoint ranges (in which case we say that the $s_i$ are pairwise orthogonal), then  $\sum_{i=1}^k s_k\in \llbracket R\rrbracket$ is defined to be the partial isomorphism which coincides with $s_i$ on $\dom s_i$ and is undefined elsewhere. 
Given $F\subseteq \llbracket R\rrbracket$ we write $\mathbf{\Sigma}F$ for the set of all elements in $\llbracket R\rrbracket$ which can be written as a finite sum of elements in $F$. That is, $\mathbf{\Sigma}F$ is the set of all sums $\sum_{i=1}^k s_k$ where $s_1 , \dots ,s_k$ are pairwise orthogonal elements of $F$. 

The equivalence relation $R$ is said to \emph{preserve the measure $\mu$} if $\mu(\dom s)=\mu(\ran s)$ for all $s\in \llbracket R\rrbracket$. By ``pmp equivalence relation on $(X,\mu)$" we mean a measured equivalence relation on the probability space $(X,\mu)$ which preserves the measure $\mu$. 

Let $R$ be a pmp equivalence relation on $(X,\mu)$. The \emph{uniform distance} $|s-t|$ between two elements $s,t\in \llbracket R\rrbracket$ is defined by 
\[
|s-t|:=\mu\{x\in \dom s\cup \dom t\,|\, s(x)\neq t(x)\},
\] 
with the convention that $s(x)\neq t(x)$ if $x\in \dom s\mathop{\Delta} \dom t$. See Lemma \ref{L - Distance} below for some elementary properties of the distance $|\cdot|$ (note that the restriction of $|\cdot|$ to $[R]$ is the usual uniform distance with respect to which $[R]$ is a Polish group). We also set 
\[
\tau(s)=\tau_\R(s):=\mu\{x\in \dom s\,|\, s(x)= x\},
\] 
that is, the restriction to $\llbracket R\rrbracket$ of the normalized trace on the von Neumann algebra $LR$ of $\R$ (viewing elements of $\llbracket R\rrbracket$ as partial isometries in $L\R$ in the usual way).

Given pmp equivalence relations $\R$ and $\R'$, a finite set $F\subset \llbracket\R\rrbracket$, integers $n,d\geq 1$ and a $\delta>0$, we say that a map $\varphi : \llbracket\R\rrbracket \to \llbracket\R'\rrbracket$ is 
{\it $(F ,n,\delta )$-multiplicative} if 
\[
| \varphi (st ) - \varphi (s )\varphi (t ) | < \delta 
\]
for all $s, t \in \bs F_\pm^{ n}$ such that $st\in \bs F_\pm^{ n}$, and $(F ,n,\delta )$-\emph{trace-preserving} if
\[
| \tau_{\R'}(\varphi (s)) - \tau_R(s) | < \delta.
\]
for all $s\in\bs F_\pm^n$, where we write $F_{\pm}$ for the finite subset of $\llbracket\R\rrbracket$ defined by $F_{\pm}:=F\cup \{s^{-1},\, s\in F\}\cup \{1\}$, and $F^n$ for the finite subset of $\llbracket\R\rrbracket$ consisting of all products of $n$ elements of $F$, with the convention that $F_{\pm}^n$ refers to the subset $(F_{\pm})^n$.

  We note that these two notions are local in the sense that they only involve knowledge of the values of $\p$ on the finite set $\bs F_\pm^n$.

Let $d$ be an integer. We denote by $[d]$ the symmetric group on $d$ elements and let $\llbracket d\rrbracket$ be the associated inverse semigroup of partial permutations, i.e., the
inverse semigroup associated to the full equivalence relation on the set with $d$ elements, endowed with the uniform probability measure.

Given  $F$ a finite subset of $\llbracket\R\rrbracket$, $n\in\Nb$, and
 $\delta > 0$,  we define $\SA (F ,n,\delta ,d)$ to be the set of all unital maps $\varphi : \llbracket\R\rrbracket \to \llbracket d \rrbracket$ 
which are $(F ,n,\delta )$-multiplicative  and $(F ,n,\delta )$-trace-preserving.  Elements of $\SA (F ,n,\delta ,d)$ are called (sofic) \emph{microstates} for $R$.
We write  $\NSA (F ,n ,\delta ,d)$ for the number of distinct restrictions of elements of $\SA (F ,n,\delta ,d)$ to the set $F$.

\begin{definition}[Definition of the sofic dimension]\label{D-equivalence relation local}
We set
\begin{align*}
s(F ,n,\delta ) &= \limsup_{d\to\infty} \frac{1}{d \log d} \log \NSA (F ,n ,\delta ,d), \\
s(F ,n) &= \inf_{\delta > 0} s(F ,n,\delta ), \\
s(F ) &= \inf_{n\in\Nb} s(F ,n) .
\end{align*}
We similarly define $\slower (F,n,\delta )$, $\slower (F,n)$, and $\slower (F)$ by replacing the
limit supremum in the first line with a limit infimum. Given a nonprincipal ultrafilter $\omega \in \beta(\IN)\setminus \IN$, we set 
\begin{align*}
s_\omega(F ,n,\delta ) = \lim_{d\to\omega} \frac{1}{d \log d} \log \NSA (F ,n ,\delta ,d), 
\end{align*}
and define $s_\omega(F ,n)$ and $s_\omega(F )$ as above by taking infima over $\delta>0$ and $n\in\Nb$. In particular, $\slower(F)\leq s_\omega(F)\leq s(F)$. We call $s(F)$ the \emph{sofic dimension} of $F$ and $\slower(F)$ the \emph{lower sofic dimension} of $F$.
\end{definition}

We trust that our notation  $s(F)$ for  the invariant and  $s\in F$ for the element will not cause confusion. Ozawa \cite{Oz}  is using the 2-norm on $[d]$ (rather than the Hamming distance) to define the sofic property, but the resulting two definitions are easily seen to be equivalent.

\begin{definition}\label{D-gen old}
A set $F\subset \llbracket\R\rrbracket$ is called a \emph{generating set} (or \emph{graphing}) of $R$ if for almost every $(x,y)\in \R$ there exists an $n\in \IN$ and a $n$-tuple $(s_1,\ldots, s_n)\in F_{\pm}^{\times n}$ such that   $y=s_1\cdots s_n(x)$. An equivalence relation is said to be \emph{finitely generated} if it admits a finite generating set.
\end{definition}

For example, the Connes--Feldman--Weiss theorem \cite{Connes1981a} states that every amenable equivalence relation $R$ is \emph{singly generated}, namely generated by a single transformation $s$ in  $[R]$. In other words, for almost every $(x,y)\in \R$, we can find an $n\in \IZ$ such that $y=s^n(x)$.  

Levitt introduced the notion of  \emph{cost} for generating a relation $R$, which was much studied recently (see \cite{Lev,Gaboriau2000a}) and which we recall now.  Given a countable subset $F\subset  \llbracket\R\rrbracket$, set
$\cost(F):=\sum_{s\in F} \mu(\dom s)$. This is usually not an orbit equivalence invariant, but we can define one as follows.

\begin{definition}\label{D-cost equivalence relation} Given a pmp equivalence relation $R$, set
\begin{align*}
\cost (\R) &= \inf_{F} \cost(F) 
\end{align*}
where the infimum is taken over all countable generating sets of $\R$.
\end{definition}

Let us now introduce a different notion of generating set for equivalence relations, and define the notion of $s$-regularity.

\begin{definition}\label{D-gen}
A set $F\subseteq \llbracket R\rrbracket$ is called a \emph{dynamical generating set}  for $R$ if for every $t\in \llbracket R\rrbracket$ and $\varepsilon > 0$
there are an $n\in\Nb$ and $s\in \mathbf{\Sigma}F_\pm^n$ such that $| t - s | < \varepsilon$. We say that $R$ is \emph{dynamically finitely generated} if it admits a finite dynamical generating set.
\end{definition}

\begin{example}\label{Ex-bernoulli}
Let $\IZ\acts \{0,1\}^\IZ$ be a Bernoulli action of $\IZ$, where $\{0,1\}^\IZ$  is endowed with an invariant product probability measure. Let $s$ be the automorphism of $\{0,1\}^\IZ$ corresponding to the generator of $\IZ$ and, for $i=0,1$, let $p_i$ be projection onto the Borel subset of $\{0,1\}^\IZ$ consisting of all sequence whose 0 coordinate is $i$. Then $F=\{s, p_0 , p_1\}$ is a dynamical generating set for $\IZ\acts \{0,1\}^\IZ$. More generally, if $\G\acts (X,\mu)$ is a pmp action of a finitely generated group which admits a finite generating partition $\sP$, then the union of the finite set of projections associated to $\sP$ and a finite generating set of $\G$ (viewed as a subset of $\Aut(X,\mu)$)  forms a dynamical generating set for the orbit equivalence relation associated to $\G\acts (X,\mu)$. 
\end{example}

\begin{proposition}
An ergodic pmp equivalence relation is finitely generated if and only if it is dynamically finitely  generated.
\end{proposition}

\begin{proof} It is clear that any dynamical generating set is a generating set in the sense of Definition \ref{D-gen old}.
For the reverse implication, suppose that the pmp equivalence relation $R$ is finitely generated. Then there are
partial transformations $s_1 , \dots , s_n$ which generate $R$   in the sense of Definition \ref{D-gen old} and such that $s_1$ is an ergodic
automorphism of $(X,\mu )$. Indeed, we may just take any finite generating set of $R$ and add to it an ergodic transformation in $[R]$, whose existence is guaranteed by the ergodicity assumption on $R$ (see e.g. \cite{Kechris2004}). By Dye's theorem \cite{Dye1963}, the automorphism $s_1$ is orbit equivalent to a Bernoulli shift 
$\Zb\curvearrowright \{ 0,1 \}^\Zb$. Thus, as in the previous example, we can find a dynamical generating set $F$ of the subrelation of $R$ generated by $s_1$. Then the set $F\cup \{s_2 , \dots , s_n \}$
is a dynamical generating set for $R$.
\end{proof}

We will also need the following simple lemma.

\begin{lemma}\label{L - 2 dynamical generating sets}
Assume that the equivalence relation $R$ is generated by two subrelations $R_1$ and $R_2$, and that $F_1$ and $F_2$ are dynamical generating sets of $R_1$ and $R_2$. Then $F_1\cup F_2$ is a dynamical generating set of $R$.
\end{lemma}

We now introduce the notion of regularity  for pmp equivalence relations mentioned in the introduction. This notion is comparable to regularity in the context of free entropy, as defined in  \cite[Definition 3.6]{Voiculescu1993} and \cite[Definition 2.1]{Brown2006e}.

\begin{definition}\label{D-s regular}
Let $F\subset \llbracket\R\rrbracket$ be a finite set. The set $F$ is said to be \emph{regular} if $\slower(F)=s(F)$.  A finitely generated equivalence relation  $R$ is said to be \emph{$s$-regular} if it admits a finite dynamical  generating set which is regular.
\end{definition}

\begin{remark}
It is a corollary of Theorem \ref{P-generating} that if the equivalence relation $R$ admits a regular finite dynamical generating set, then all finite dynamical generating sets for $R$ are regular. This is the case for example if $R$ is amenable, or more generally if $R$ is treeable (see Corollary \ref{C-treeable}). We do not know of an example of a finite generating set of an equivalence relation which is not regular.
\end{remark}

\section{Preliminary technical lemmas}\label{S-Preliminary}

 In this section we first establish several lemmas that will be used in the course of proving Theorem \ref{th1} and Theorem \ref{th2}.

\begin{lemma}\label{L - Distance} Let $R$ be a pmp equivalence relation, and let $r,s,t\in \llbracket\R\rrbracket$. Then
\begin{enumerate}
\item $|s-t|=\mu(\dom s\vartriangle \dom t) + \tau(s^{-1}st^{-1}t)-\tau(st^{-1})$
\item[] \hspace{1.2cm}$=\tau(s^{-1}s)+\tau(t^{-1}t)- \tau(s^{-1}st^{-1}t)-\tau(st^{-1})$,
\item if $s^{-1}st^{-1}t=0$, then $|s-t|=\mu(\dom s)+\mu(\dom t)$,
\item $|s-t|=|s^{-1}-t^{-1}|$,
\item $|\tau(s)-\tau(t)|\leq |s-t|$,
\item $|rs-rt|=|ps-pt|\leq |s-t|$ where $p=r^{-1}r$,
\item $|sr-tr|=|sp-tp|\leq |s-t|$ where $p=rr^{-1}$.
\end{enumerate}
Furthermore, if for some $\delta>0$ we have
\begin{enumerate}
\setcounter{enumi}{6}
\item $|s-p|<\delta$, where $p$ is a projection, then there exists a projection $p'\leq s^{-1}s$ such that $|p-p'|<\delta$. 
\item $|sts-s|<\delta$ and $|tst-t|<\delta$, then $|t-s^{-1}|<3\delta$.
\end{enumerate}
\end{lemma}

This lemma is elementary and we leave the proof to the reader.  Let us establish
(8) for example. Let $p=ts$, so that $\dom p\subset \dom s$.  Since by assumption $|sp-s|<\delta$ we have, for all $x\in \dom s$ outside a subset of measure at most $\delta$, that $sp(x)=s(x)$. In particular $x\in \dom p$  and $p(x)=x$, and thus $|s^{-1}s-p|<\delta$. Similarly, $|ss^{-1}-st|<\delta$.
Now the second inequality shows that $|pt-t|<\delta$ and so we conclude that
\[
|t-s^{-1}|<|pt-s^{-1}|+\delta\leq |s^{-1}st-s^{-1}|+2\delta=|st-ss^{-1}|+2\delta<3\delta.
\]

The next lemma shows automatic (approximate) continuity of microstates. 

\begin{lemma}\label{L - sofic approx distance} Let $R$ be a pmp equivalence relation. Fix a finite set $F\subset \llbracket\R\rrbracket$, integers $n\geq 3,d\geq 1$, a $\delta>0$, and a microstate $\p\in \SA(F,n,\delta,d)$. Then we have
\begin{enumerate}
\item $|\p(s^{-1})-\p(s)^{-1}|\leq 6\delta$ for all nonzero $s\in \bs F_{\pm}^{\lfloor n/2\rfloor}$,
\item $|\p(\prod_{i=1}^k s_i^{\e_i})-\prod_1^{k} \p(s_i)^{\e_i}|\leq (k+6k_0)\delta$ for all nonzero $s_1,\ldots, s_k\in \bs F_\pm^{\lfloor n/k\rfloor}$ and $(\e_1,\ldots,\e_k)\in \{\pm1\}^{\times k}$, where $2\leq k\leq n/2$ and $k_0$ is the number of indices $i$ such that $\e_i=-1$,
\item $|\p(s)-\p(t)|\leq |s-t|+40\delta$ for all $s,t\in \bs F_{\pm}^{\lfloor n/4\rfloor}$.
\end{enumerate}
\end{lemma}

\begin{proof} 

(1) Let $s\in\bs F_{\pm}^{{\lfloor n/2\rfloor}}$ so that $s^{-1}\in\bs F_{\pm}^{\lfloor n/2\rfloor}$. Since $ss^{-1}s=s$ and $s^{-1}ss^{-1}=s^{-1}$, we have
\[
|\p(s)-\p(s)\p(s^{-1})\p(s)|<2\delta ~~~\mathrm{and}~~~|\p(s^{-1})-\p(s^{-1})\p(s)\p(s^{-1})|<2\delta.
\]
Applying Lemma \ref{L - Distance}, we see that $|\p(s^{-1})-\p(s)^{-1}|<6\delta$. 
Assertion (2) is immediate. Let us show (3). For all $s,t\in \bs F_{\pm}^{n/4}$ we have
\[
|s-t|=\tau(s^{-1}s)+\tau(t^{-1}t)- \tau(s^{-1}st^{-1}t)-\tau(st^{-1})\\
\]
and
\begin{align*}
|\p(s)-\p(t)|&=\tr(\p(s)^{-1}\p(s))+\tr(\p(t)^{-1}\p(t))\\
&\hspace{1.5cm}- \tr(\p(s)^{-1}\p(s)\p(t)^{-1}\p(t))-\tr(\p(s)\p(t)^{-1})
\end{align*}
Hence 
\[
|\p(s)-\p(t)| \leq |s-t| +8\delta+8\delta+16\delta+8\delta=|s-t| +40\delta.
\]
This proves the lemma.
\end{proof}

 Let $\R$ be a pmp equivalence relation on $(X,\mu)$.  It will be convenient to extend the additive structure on $\llbracket\R\rrbracket$ to any $k$-tuple of elements. This can be done as follows (note that this extension does not coincide anymore with the additive structure on the von Neumann algebra $LR$).

\begin{definition}\label{D -sum} Let $s_1,\ldots, s_k\in \llbracket\R\rrbracket$ be partial isomorphisms. We define $\sum_{i=1}^k s_i$ to be the  partial isomorphism
 \[
\sum_{i=1}^k s_i:=\sum_{i=1}^k s_i\pi_i(s_1,\ldots,s_n)
 \]
where  $\pi_i(s_1,\ldots,s_n)$ is the projection defined by 
\[
\pi_i(s_1,\ldots,s_n)=\left(s_i^{-1}s_i\prod_{j\neq i} (1-s_j^{-1}s_j)\right)\times s_i^{-1}\left(s_is_i^{-1}\prod_{j\neq i} (1-s_js_j^{-1})\right)s_i.
\]  
\end{definition}

It is clear that this extends the definition of addition in the case that the $s_i$ are pairwise orthogonal.

Fix a finite set $F\subset \llbracket\R\rrbracket$, integers $n,d\geq 1$, a $\delta>0$. 
\begin{lemma}\label{L - sum-proj} 
Let $s_1,\ldots, s_k\in \bs F_\pm^{\lfloor n/4\rfloor}$  be pairwise orthogonal elements, and let $\p\in \SA(F,n,\delta,d)$ be a microstate. Then we have
\[
|\p(s_i)\pi_i(\p(s_1),\ldots,\p(s_k))-\p(s_i)|<40(k-1)\delta
\]
\end{lemma} 

\begin{proof}
Since $s_i^{-1}s_is_j^{-1}s_j=0$ we have using Lemma \ref{L - sofic approx distance} that 
\[
\tr(\p(s_i)^{-1}\p(s_i)\p(s_j)^{-1}\p(s_j))<16\delta.
\]
Using a similar inequality for $s_is_i^{-1}s_js_j^{-1}$, we deduce that 
\[
\tr(\pi_i(\p(s_1),\ldots,\p(s_k)))>\tr(\p(s_i)^{-1}\p(s_i))-32(k-1)\delta,
\] 
and the lemma follows.
\end{proof}

 We now prove automatic (approximate) linearity of microstates. 
 
\begin{lemma}\label{L -Linearity}
Let $\p\in \SA(F,4n,\delta,d)$. Then for each $s\in \mathbf{\Sigma}F^n_\pm$ with decomposition $s=\sum_{i=1}^ks_i$ with the $s_i\in \mathbf{\Sigma}F^n_\pm$ pairwise orthogonal, we have
\[
\bigg |\p(s) - \sum_{i=1}^k\p(s_i)\bigg|< 150(2|F|+1)^{2n} \delta.
\]
\end{lemma}

\begin{proof} Write $\pi_i=\pi_i(\p(s_1),\ldots,\p(s_k))$ so that  $|\pi_i-\p(s_i)^{-1}\p(s_i)|<
40(k-1)\delta$ using the previous lemma. Given that $|\p(s_i)-\p(s)\p(s_i)^{-1}\p(s_i)|<8\delta$ we obtain 
\[
|\p(s)\pi_i-\p(s_i)|<50k\delta.
\]
In particular 
\[
\bigg|\p(s)\pi -\sum_{i=1}^k\p(s_i)\bigg|<50k^2\delta
\]
where $\pi:= \sum_{i=1}^k\pi_i$ and we recall that $\sum_{i=1}^k\p(s_i)$ means $\sum_{i=1}^k\p(s_i)\pi_i$. On the other hand, 
\begin{align*}
\tau(\p(s)^{-1}\p(s)\pi)&=\sum_{i=1}^k\tau(\p(s)^{-1}\p(s)\pi_i)\\
&\hspace{-.5cm}>\sum_{i=1}^k\tau(\p(s_i)^{-1}\p(s_i))-50k^2\delta>\sum_{i=1}^k\tau(s_i^{-1}s_i)-60k^2\delta\\
&\hspace{2cm}=\tau(s^{-1}s)-60k^2\delta>\tau(\p(s)^{-1}\p(s))-70k^2\delta
\end{align*}
so that 
\[
\bigg|\p(s) -\sum_{i=1}^k\p(s_i)\bigg|<150k^2\delta
\]
establishing the lemma.
\end{proof}

 Let $\R$ be a pmp equivalence relation on $(X,\mu)$. 
Given a finite set $F\subset  \llbracket\R\rrbracket$ we define
a pseudometric on the set of all unital linear maps $ \llbracket\R\rrbracket \to \llbracket d \rrbracket$  by 
\[
|\varphi - \psi |_F := \max_{s\in F} | \varphi (s) - \psi (s) |. 
\]
Let $\e\geq 0$. We write $N_\e(\SA(F,n ,\delta ,d))$ for the $\varepsilon$-covering number of $\SA(F,n ,\delta ,d)$
with respect to $|\cdot|_{F}$, namely, the minimal number of $\e$-balls required to cover $\SA(F,n ,\delta ,d)$. Note that $\NSA(F,n ,\delta ,d)=N_0(\SA(F,n ,\delta ,d))$. We then  set
\begin{align*}
s_\varepsilon (F ,n,\delta ) &= \limsup_{d\to\infty} \frac{1}{d \log d} \log N_\e(\SA (F ,n ,\delta ,d)), \\
s_\varepsilon (F ,n) &= \inf_{\delta > 0} s_\varepsilon (F ,n,\delta ),\\
s_\varepsilon(F) &= \inf_{n\in \IN} s_\varepsilon (F ,n).
\end{align*}
We similarly define $\e$-covering constants for $\slower$ and $s_\omega$.

\begin{lemma}\label{L-nbhd}
Let $\kappa > 0$. Then there is an $\varepsilon > 0$ such that
\[
| \{ t\in \llbracket d \rrbracket : | t-s | < \varepsilon \} | \leq d^{\kappa d} .
\]
for all $d\in\Nb$ and $s\in \llbracket d \rrbracket$.
\end{lemma}

\begin{proof}
Let $\varepsilon > 0$. Let $d\in\Nb$ and $s\in \llbracket d \rrbracket$. 
For every $t\in \llbracket d \rrbracket$ satisfying $| s-t | < \varepsilon$
the cardinality of the set of all $j\in \{ 1,\dots ,d \}$ such that $t(j) \neq s(j)$ is at most $\varepsilon d$.
Thus the set of all $t\in \llbracket d \rrbracket$ such that $| t-s | < \varepsilon$ has cardinality at most
$\sum_{j=0}^{\lfloor \varepsilon d \rfloor} \binom{d}{j} j!$. 
This sum is bounded above by
$\lfloor \e d\rfloor \binom{d}{ \lfloor\e d\rfloor } ( \lfloor\e d\rfloor )!
 \leq \e d^{1+\e d }$, which for small enough $\e$ is less that
 $d^{\kappa d}$, independently of $d$.
\end{proof}

\begin{lemma}\label{L-ineq}
Let $F$ be a finite subset of $\llbracket R\rrbracket$. Let $\kappa > 0$. Then there is an $\varepsilon > 0$ such that
\[ 
s(F ,n) \leq s_\varepsilon (F ,n) + \kappa
\]
for all $n\in\Nb$. The same inequality holds for $\slower$ and $s_\omega$.
\end{lemma}

\begin{proof}
This is a straightforward consequence of Lemma~\ref{L-nbhd}.
\end{proof}

\begin{lemma}\label{L-ineq lim sup epsilon}
Let $F$ be a finite subset of $\llbracket R\rrbracket$. Then
\[
s(F)=\lim_{\e\to 0} s_\e(F).
\]
The same equality holds for  $\slower$ and $s_\omega$.
\end{lemma}

\begin{proof}
Since $N_\e(\cdot)$ is increasing as $\e$ decreases, $s_\e(F)$ is increasing as $\e\to0$.
By taking an infimum over $n$, the previous lemma shows that for all $\kappa>0$ there is an $\e>0$ such that 
\[ 
s(F) \leq s_\varepsilon (F) + \kappa.
\]
Thus $s(F) \leq \lim_{\e\to 0} s_\varepsilon (F)$. The other inequality is clear.
\end{proof}

\section{invariance under orbit equivalence}\label{S-OE section}

\begin{theorem}\label{P-generating} Let $\R$ be a pmp equivalence relation and
let $E$ and $F$ be finite dynamical generating sets. Then $s(E ) = s(F )$, $\slower(E)=\slower(F)$, and $s_\omega(E)=s_\omega(F)$.
\end{theorem}

\begin{proof} Let us show the first equality.
By symmetry it suffices to prove that $s(E ) \leq s(F )$. Let $\e > 0$ and let $\delta>0$ be smaller than $\varepsilon /4$.
Since $F$ is generating, there exists an $n_0$ such that for all $n\geq n_0$ and every $s\in E$, there is an $s'\in  \mathbf{\Sigma}F_\sym^{ n}$ for which 
$| s - s' | < \delta$. Take a $\delta' > 0$ such that $50\delta' <\delta$ and $\delta' < \varepsilon /42000(2|F|+1)^{2n}$.
Since $E$ is generating there exists an $m\in\Nb$ such that for all $s\in \mathbf{\Sigma} F_\sym^{8 n}$, there exists an element $\theta(s)\in  \mathbf{\Sigma}E_\sym^{ m}$ for which $|s-\theta(s)|<\delta'/3$. Observe that the map $\theta: s\mapsto \theta(s)$  (extended arbitrarily to $\llbracket R\rrbracket$) is $(F,4n,\delta')$-multiplicative. Indeed,
for $s,t\in \bs F_{\pm}^{4n}$ we have
\begin{align*}
| \theta (st ) - \theta (s )\theta (t ) |&\leq | \theta (st) - st |+ |st - \theta (s )\theta (t ) |\\&\hspace*{15mm} \
<\delta'/3+ | \theta (s) -  s  |+| \theta (t) -  t  | <\delta'.
\end{align*}

Let $\varphi\in\SA ( E ,4mn,\delta' ,d)$ and write $\varphi^\natural$ for $\varphi\circ\theta$.  We note  first that $\varphi^\natural \in \SA (F ,4n,50\delta' ,d)$. 
Indeed, given $s,t \in \bs F_\sym^{4n}$ we have by Lemma \ref{L - sofic approx distance},
\begin{align*}
| \varphi^\natural (st) - \varphi^\natural (s ) \varphi^\natural (t ) |&= | \varphi (\theta (st)) - \p(\theta (s ) \theta (t )) | \\
&\hspace*{15mm} \ + | \varphi (\theta (s ) \theta (t )) -  \varphi (\theta (s ))  \varphi (\theta (t )) | \\
&\leq | \theta (st) - \theta (s ) \theta (t ) | + 40 \delta' + \delta' \\
&< 50\delta'. 
\end{align*}
and
\begin{align*}
|\tr\circ\varphi^\natural (t) - \tau (t) | 
&\leq | \tr\circ\varphi (\theta (t)) - \tau (\theta (t)) | + | \tau (\theta (t)) - \tau(t)  | \\
&< \delta' + | \theta (t) - t | < 2\delta'.
\end{align*}

Now consider microstates  $\varphi$ and $\psi$ in $\SA ( E ,4mn,\delta' ,d)$ and suppose that $\p^\natural$ and $\psi^\natural$ coincide on $F$. Then  $|\p^\natural(s^{-1})-\psi^\natural(s^{-1})|<6\delta'$ for all $s$ in $\bs F_{\pm}^{n}$  by Lemma \ref{L - sofic approx distance} and thus,  for all $t_1 , \dots ,t_n \in F_\sym$ 
we have, writing $t=t_1\cdots t_n$,
\begin{align*}
| \p^\natural (t ) - \psi^\natural (t) |
&\leq \bigg| \p^\natural (t) - \prod_{i=1}^n\p^\natural (t_i )\bigg| \\
&\hspace{2cm} + \bigg|  \prod_{i=1}^n\p^\natural (t_i ) -  \prod_{i=1}^n\psi^\natural (t_i ) \bigg|+ \bigg|  \prod_{i=1}^n\psi^\natural (t_i ) - \psi^\natural (t )\bigg | \\
&< n\delta' + \sum_{i=1}^n | \p^\natural (t_i ) - \psi^\natural (t_i ) | + n\delta' \\
&< 8n\delta'. 
\end{align*}
If $t=\sum_{i=1}^k t_i\in  \mathbf{\Sigma}F_\pm^n$ is a pairwise orthogonal decomposition, then by Lemma \ref{L - sum-proj} we have
\[
|\p^\natural(t_i)\pi_i(\p^\natural(t_1),\ldots,\p^\natural(t_k))-\p^\natural(t_i)|<40\times 50(k-1)\delta'\leq2000(2|F|+1)^n\delta',
\]
and similarly
\[
|\psi^\natural(t_i)\pi_i(\psi^\natural(t_1),\ldots,\psi^\natural(t_k))-\psi^\natural(t_i)|<2000(2|F|+1)^n\delta'
\]
whence, by Lemma \ref{L -Linearity},
\begin{align*}
| \p^\natural (t ) - \psi^\natural (t) |&\leq\bigg|\sum_{i=1}^k\p^\natural(t_i)\pi_i(\p^\natural(t_1),\ldots,\p^\natural(t_k))-\sum_{i=1}^k\psi^\natural(t_i)\pi_i(\psi^\natural(t_1),\ldots,\psi^\natural(t_k))\bigg|\\
&\hspace{5cm} + 2\times 150\times 50(2|F|+1)^{2n} \delta'\\
&\leq \sum_{i=1}^k|\p^\natural(t_i)\pi_i(\p^\natural(t_1),\ldots,\p^\natural(t_k))-\psi^\natural(t_i)\pi_i(\psi^\natural(t_1),\ldots,\psi^\natural(t_k))|\\
&\hspace{5cm} +15000(2|F|+1)^{2n} \delta'\\
&\leq (2000(2|F|+1)^n\delta'+8n\delta'+2000(2|F|+1)^n\delta')k\\
&\hspace{5cm} +15000(2|F|+1)^{2n} \delta'\\
&\leq 20000(2|F|+1)^{2n}\delta'.
\end{align*}
Let $s\in E$. We can find $s'\in  \mathbf{\Sigma}F_\sym^{ n}$ such that $| s - s' | < \delta$. Then $|\theta(s')-s'|<\delta'$, and thus $|s-\theta(s')|<\delta+\delta'$. Hence,
\[
|\p(s)-\p^\natural(s')|=|\p(s)-\p(\theta(s'))|<|s-\theta(s')|+ 40\delta'<\delta+40\delta' 
\]
Similarly, $|\psi(s)-\psi^\natural(s')|<\delta+40\delta' $, so that
\begin{align*}
 |\varphi(s) - \psi(s) |&\leq |\p(s)-\p^\natural(s')|+| \p^\natural (s') - \psi^\natural (s') |+ |\psi(s)-\psi^\natural(s')|\\
 &<2\delta+80\delta'+20000(2|F|+1)^{2n}\delta'\\
 &<\e/2+21000(2|F|+1)^{2n}\delta'<\e.
\end{align*}
Consequently, $|\varphi - \psi |_{E}<\e$ and thus
\begin{align*}
N_\e(\SA ( E ,4mn,\delta' ,d))\leq  \NSA (F ,n,50\delta' ,d)\leq  \NSA (F ,n,\delta ,d),
\end{align*}
By taking a limit supremum over $d$, an infimum over $\delta>0$ and over $n\in\Nb$, we conclude that $s_{\e}( E) \leq s(F)$. Hence, Lemma \ref{L-ineq lim sup epsilon} shows that $s( E )\leq s(F )$. 

The same proof applies to $\slower$ and $s_\omega$.
\end{proof}

In view of Theorem \ref{P-generating}, we introduce the following isomorphism invariant for dynamically finitely generated  (e.g. finitely generated ergodic) equivalence relations. 

\begin{definition}
Let $R$ be a pmp equivalence relation on $(X,\mu)$. Assume that $R$ is dynamically finitely generated and let $F$ be a finite dynamical generating set. Then we set
\[
s(R):=s(F)
\] 
and call this value the \emph{sofic dimension} of $R$. we define similarly $\slower(R)$ and $s_\omega(R)$. 
\end{definition}

\begin{remark}
It is possible to extend the definition of $s(R)$ to pmp equivalence relation which are not finitely generated: see \cite[Definition 2.3]{dkp}.
\end{remark}

We end this section with the  proof of  the formula $s(R)\leq \cost(R)$ mentioned in the introduction.

\begin{lemma}\label{L - cost} Let $R$ be a pmp equivalence relation on $(X,\mu)$ and let $F$ be a finite subset of $ \llbracket\R\rrbracket$. Then $s(F)\leq \cost(F)$.
\end{lemma}

\begin{proof}
We show that $s(F,4)\leq \cost(F)$, from which the lemma follows readily. Let $\delta>0$ and $d\in \IN$. Take a $\p\in \SA(F,4,\delta,d)$. Then for every $s\in F$ we have by Lemma \ref{L - sofic approx distance}
\begin{align*}
|\tr(\p(s)\p(s)^{-1})-\tau(ss^{-1})|&<|\tr(\p(s)\p(s)^{-1})-\tr(\p(ss^{-1})|\\
&\hspace{2cm}+|\tr(\p(ss^{-1}))-\tau(ss^{-1})|\\
&<8\delta+\delta=9\delta.
\end{align*}
Thus, we have $|F|$ partial permutations $\p(s)$,  $s\in F$, of the set with $d$ elements whose respective domains $A_s$ and ranges $B_s$ have cardinality $(\cost(s)\pm 9\delta)d$. Since we have  $\binom d {|A_s|}$ ways to choose $A_s$ and $B_s$, and $|A_s|!$ ways to choose a bijection from $A_s$ to $B_s$, we obtain
\[
\NSA(F,4,\delta,d)\leq \prod_{s\in F} {\binom d {|A_s|}}^2|A_s|!=\prod_{s\in F} \frac{d!^2}{|A_s|!(d-|A_s|)!^2}
\]
An easy computation using  the Stirling formula shows that, taking the limit supremum over $d$,
\begin{align*}
s(F,4,\delta)&\leq \sum_{s\in F} 2 - (\cost(s)- 9\delta) - 2(1-\cost(s)- 9\delta)\\
&< \cost(F)+30|F|\delta.
\end{align*}
The claim follows by taking the infimum over $\delta$.
\end{proof}

\begin{proposition}
Let $R$ be a finitely generated ergodic pmp equivalence relation on $(X,\mu)$. Then 
$s(R)\leq \cost(R)$.
\end{proposition}
\begin{proof}
Let $\e>0$. Since $R$ is ergodic, we may choose a generating set $F$ of $R$ (in the sense of Definition \ref{D-gen old}) with 
\[
\cost(F)<\cost(R)+\e
\]
which contains an ergodic automorphism of $(X,\mu)$ (see e.g. \cite[Lemma III.5]{Gaboriau2000a}). Since this automorphism is orbit equivalent to a Bernoulli shift (by Dye's theorem), we may replace it in $F$ by two partial isomorphisms of $(X,\mu)$ of cost $\frac 1 2$ each, which generate the same subrelation of $R$ (compare Example \ref{Ex-bernoulli}). The resulting generating set $F'$ has the same cost as that of $F$, and is dynamically generating. Thus 
\[
s(R)=s(F')\leq \cost(F')\leq \cost(R)+\e.
\] 
Hence the result by taking an infimum over $\e>0$.
\end{proof}

\section{Microstates and finite inverse semigroups}\label{S-amenable}

In this section $G$ denotes a finite inverse semigroup, which we assume to be principal for convenience. Namely, $G$ is of the form $G= \llbracket R_G\rrbracket$ where $R_G$ is an equivalence relation on a finite set $X_G$. The set $X_G$ can be taken to be the set of minimal projections in $G$, in which case $R_G$ is the equivalence relation on $X_G$ generated by von Neumann equivalence of projections: $p\sim q$ if and only if $p=s^{-1}s$ and $q=ss^{-1}$ for some $s\in G$. We endow $X_G$ with any invariant probability measure with rational values, and $G$ with the corresponding tracial state $\tau$.

Given a generating set $E$ of $G$,  a pmp equivalence relation $R$ on a probability space $(X,\mu)$, and a unital trace-preserving embedding $G\subset  \llbracket R\rrbracket$, and a finite subset $F\subset \llbracket \R\rrbracket$ containing $E$,
we denote by $\SA_G(F,n ,\delta ,d)$ the set of all maps $\p\in \SA(F,n ,\delta ,d)$ for which the restriction $\p_{|G}$ is a trace-preserving embedding. We write $\NSA_G(F,n ,\delta ,d)$ for the number of restrictions to $F$ of elements of $\SA_G(F,n ,\delta ,d)$.

\begin{lemma}\label{G equiv} There exist  integers $n_G,m_G,d_G$ depending only on $G$ such that, for every pmp measured equivalence relation $\R$,  generating set  $E$ of $G$,  unital trace-preserving embedding $G\subset \llbracket \R\rrbracket$, finite subset $F\subset \llbracket \R\rrbracket$ containing $E$,  $\e\geq 0$,  $\delta> 0$,  integers $n$ and  $d>\delta^{-1}$, we have
\[
 N_{m_G\delta+\e}(\SA (F,n_G+n,\delta ,d_Gd)) \leq N_\e(\SA_G(F,n_G+n,m_G\delta,d_Gd)).
 \]
In particular,
\[
 s(F) = \inf_{n\in \mathbb{N} }\inf_{\delta>0} \limsup_{d\to\infty} \frac{1}{d_Gd \log d_Gd} \log \NSA_G (F,n ,\delta ,d_Gd).
 \]
The number $d_G$ can be chosen to be any integer such that, for any $d\in \IN$, there exists a unital trace-preserving embedding $G\subset  \llbracket d_Gd\rrbracket$. Furthermore, the same results hold for $\underline s(F)$ and $s_\omega(F)$  provided one  replaces the limit supremum by $\liminf_{d\to\infty}$  and $\lim_{d\to\omega}$ respectively.
\end{lemma}

\begin{proof}  
Let $d_G$ be such that  $d_G\tau(p)\in \IN$ for all $p\in X_G$, and let $n_G$ be  at least 4 times the length of the longest reduced word over any generating set of $G$. We set $t_G=\min_{p\in X_G} \tau(p)$ and let $m_G$ be a integer larger than $100n_G|G|^3|X_G|t_G^{-1}+10$.
Take $E,F,n,\delta,d$ as in the statement of the lemma, and consider for  all $s,t\in G$ and $\p\in \SA (F,n_G+n,\delta/m_G ,d_Gd)$ the set
\[
V_{s,t}^\p:=\{i\in \{1,\ldots, d_Gd\}\mid\,\p(st^{-1})i=\p(s)\p(t)^{-1}i\}.
\]
By the definition of $n_G$ we have $|\p(st^{-1})-\p(s)\p(t)^{-1}|<7\delta$, and thus  
\[
\left |V_{s,t}^\p\right|> (1-7\delta)d_Gd
\] 
for all $s,t\in G$. Denoting $V^\p:=\bigcap_{s,t\in G} V_{s,t}^\p$, it follows that $|V^\p|>(1-7|G|^2\delta)d_Gd$.
Let $p_\p$ be the projection onto the subset $\{i\in V^\p \mid \p(G)i\subset V^\p\}$ of $V^\p$. We have $\tr(p_\p)>1-7|G|^3\delta$ and   the restriction to $G$ of the compression 
$p_\p\p p_\p : \llbracket \R\rrbracket\to p_\p\llbracket d_Gd\rrbracket p_\p$
of $\p$ is a morphism of inverse semigroups.  For any $p\in X_G$ we have
\[
\tr(p_\p\p(p))>\tr(\p(p))-7|G|^3\delta>\tau(p)-8|G|^3\delta.
\]
Thus, there exists a projection $p'\leq p_\p\p(p)$ such that $\tr(p')>\tau(p)-8|G|^3\delta$. Furthermore, by replacing each $p'$ with $p'\times \prod_{q\neq p}q'$ if necessary, we may assume that the projections $p'$ satisfy $p'q'=0$ for all $p\neq q\in X_G$. In doing so we obtain 
\[
\tr(p')>\tau(p)-8|G|^3|X_G|\delta\geq \tau(p)(1-8|G|^3|X_G|t_G^{-1}\delta).
\] 
Let $k$ be the largest integer such that $k < (1-8|G|^3|X_G|t_G^{-1}\delta)d$. Since $d_G\tau(p)\in \IN$, we have $d_G d\left (\frac k d\tau(p)\right)\in \IN$ for all $p\in X_G$, and therefore we may assume, by further restricting $p'$ if necessary, that 
 $\tr(p')=\frac k d\tau(p)$ for all $p\in X_G$. In particular the restriction to $G$ of the compression
\[
p_\p'\p p_\p' : \llbracket \R\rrbracket\to p_\p'\llbracket d_Gd\rrbracket p_\p',
\]
where $p_\p':=\sum_{p\in X_G} p'$, is a trace scaling isomorphism and,  since we have $k\geq (1-8|G|^3|X_G|t_G^{-1}\delta)d-1$, 
\[
\tr(p_\p')\geq \sum_{p\in X_G}(1-8|G|^3|X_G|t_G^{-1}\delta-1/d)\tau(p)\geq 1-8|G|^3|X_G|t_G^{-1}\delta-1/d.
\]
Furthermore, by construction $d_Gd\tr(1-p_\p')\in \IN$ is a multiple of $d_G$, so that we can choose a trace scaling isomorphism $\tilde \p$ between $G$ and an inverse semigroup of $(1-p_\p')\llbracket d_Gd\rrbracket (1-p_\p')$ such that $\tr(\tilde \p(p))=(1-\frac k d)\tau(p)$ for all  $p\in X_G$. 

Consider then the map $\p^\natural$ defined by
\begin{align*}
 \p^\natural(s)&:=\p(s)~~\mathrm{ if }~~s\in \llbracket R\rrbracket\setminus G,\\
 \p^\natural(s)&:=p_\p'\p p_\p'+\tilde \p~~\mathrm{ if }~~s\in G.
\end{align*}
It is clear that the restriction of $\p^\natural$ to $G$ is a trace-preserving isomorphism. We now show that $\p^\natural \in \SA_G(F,n_G+n,m_G\delta,d_Gd)$. Note first that for all $s\in G$
\begin{align*}
|\p(s)-\p^\natural(s)|&\leq |(1-p_\p') \p(s) (1-p_\p')-(1-p_\p') \p^\natural(s) (1-p_\p')|\\
&\hspace{1cm}+|(1-p_\p') \p(s) p_\p'|+|p_\p' \p(s) (1-p_\p')|\\
&\leq 24|G|^3|X_G|t_G^{-1}\delta+3/d.
\end{align*}
Therefore, 
\[
|\p-\p^\natural|_{\llbracket R\rrbracket}\leq 24|G|^3|X_G|t_G^{-1}\delta+3/d.
\] 
Given $s,t\in\bs F_{\pm}^{n_G+n}$ we have
\begin{align*}
| \p^\natural(st)-\p^\natural(s)\p^\natural(t)|&\leq| \p(st)-\p(s)\p(t) |+3|\p-\p^\natural|_{\llbracket R\rrbracket} \\
&<\delta+ 72|G|^3|X_G|t_G^{-1}\delta+9/d
\end{align*}
and $|\tr(\p^\natural(s))-\tau(s)|<\delta+24|G|^3|X_G|t_G^{-1}\delta+3/d$. Since $d>\delta^{-1}$, these estimates show that $\p^\natural\in \SA_G(F,n_G+n,m_G\delta,d_Gd)$.

Now if $\p,\psi\in \SA (F,n_G+n,\delta/m_G ,d_Gd)$ are such that $|\p^\natural-\psi^\natural|_F\leq\e$, then
\begin{align*}
|\p-\psi|_F&\leq |\p-\p^\natural|_F+|\p^\natural-\psi^\natural|_F+|\psi^\natural-\psi|_F\\
&<\e+48|G|^3|X_G|t_G^{-1}\delta+6/d<\e+m_G\delta.
\end{align*}
This shows the first assertion of the lemma. For the second, let $\e$ be any real number such that $\e>m_G\delta$. Then by the first assertion 
\begin{align*}
 N_{\e}(\SA (F,n_G+n,\delta ,d_Gd))&\leq  N_{m_G\delta}(\SA (F,n_G+n,\delta ,d_Gd)) \\
 &\leq \NSA_G(F,n_G+n,m_G\delta,d_Gd).
 \end{align*}
Taking a log, a limit infimum over $d$, and infima over $\delta$ and  $n$, we obtain
\[
s_\e(F)\leq \inf_{n\in \mathbb{N} }\inf_{\delta>0} \limsup_{d\to\infty} \frac{1}{d_Gd \log d_Gd} \log \NSA_G (F,n ,\delta ,d_Gd).
\]
Thus, by Lemma \ref{L-ineq lim sup epsilon},   
\[
s(F)\leq \inf_{n\in \mathbb{N} }\inf_{\delta>0} \limsup_{d\to\infty} \frac{1}{d_Gd \log d_Gd} \log \NSA_G (F,n ,\delta ,d_Gd).
\]
The other inequality is clear, and the case of $\slower, s_\omega$ is treated similarly.
\end{proof}

We note the following easy corollary, which can also be proved directly.

\begin{corollary}\label{L - s(G) G finite}
Let $R$ be a pmp equivalence relation on $(X,\mu)$ and $F\subset \llbracket R\rrbracket$ be a finite subset containing the identity. Assume that $F$ generates a finite inverse subsemigroup $G$ of $\llbracket R\rrbracket$ (namely, there exists $n_0\in \IN$ such that $F_\pm^{n_0}=F_\pm^n$ for all $n\geq n_0$). Let $D\subset X$ be a fundamental domain for $G$. Then 
\[
\slower(F)=s(F)=\slower(G)=s(G)=1-\mu(D).
\]  
In particular, if  $R$ has finite classes and  $D$ denotes a fundamental domain, then all finite generating subsets $F\subset \llbracket R\rrbracket$ are regular and satisfy $s(F)=1-\mu(D)$. More generally, the same assertion holds if $R$ is amenable, where $D$ is the fundamental domain of the finite component of $R_3$.
\end{corollary}
\begin{proof}
Lemma \ref{G equiv} shows that
\[
s(F)= \inf_{n\in \mathbb{N} }\inf_{\delta>0} \limsup_{d\to\infty} \frac{1}{d_Gd \log d_Gd} \log \NSA_G (F,n ,\delta ,d_Gd).
\]
Since  $F^n=G$ for all $n\geq n_0$, we need to estimate the number of unital trace-preserving embeddings of $G$ into $\llbracket d_Gd\rrbracket$. By the definition of $d_G$ there exists at least one such  embedding $\p : G \to \llbracket d_Gd\rrbracket$. Then, for any permutation $\theta\in [d_Gd]$, the conjugate $\theta\p\theta^{-1}$ is another such embedding, and it is easy to see that any such embedding is obtained in this way. Let $[d_Gd]_G$ be the subgroup of all permutations in $[d_Gd]$ which commute with $\p(G)$.
Clearly, if $\theta_1\p\theta_1^{-1}$ and $\theta_2\p\theta_2^{-1}$ coincide on $F$, then they coincide on $G$ and thus $\theta_1^{-1}\theta_2\in [d_Gd]_G$. Hence, the total number of embeddings is
\[
\frac{(d_Gd)!}{|[ d_Gd ]_G|}.
\]
Denote by $\mu_k\in [0,1]$ the measure of a fundamental domain of the subrelation of $R_3'$ consisting of all $R_3'$-classes of cardinality $k\in \IN$. Then we have $\mu(D')=\sum_{k\geq 1} \mu_k$ and $|[ d_Gd ]_G|=\prod_{k\geq 1} (\mu_kd_Gd)!$ (both the sum and the product are uniformly finite in $d$), so that $\log |[ d_Gd ]_G|\sim \mu(D') d_Gd\log d_Gd$ as $d$ tends to infinity, using Stirling's formula.
Thus we obtain
\[
s(F,n)=1-\mu(D)
\]
for all $n\geq n_0$. Taking a limit infimum instead of a limit supremum, we see that $\slower(F,n)=1-\mu(D)$, yielding the first assertion. In the case where $R$ is amenable, a direct application of the Connes--Feldmann--Weiss theorem \cite{Connes1981a} shows that, for any $\e>0$, any finite generating set $F$ and any $n$ there exists a finite principal inverse semigroup  $G\subset \llbracket R_3\rrbracket$ such that for any $s\in \bs F_\pm ^n$, there is a $g\in G$ for which $|s-g|<\e$. In addition, $G$ can be chosen to have a fundamental domain $D'$ such that $\mu(D\setminus D')<\e$.  For every  embedding $\p : G \to \llbracket d_Gd\rrbracket$, we define a map $\bs F_\pm^n\to \llbracket d_Gd\rrbracket$ mapping $s\in \bs F_\pm^n$ to $\p(g)$ where $g$ is any element of $G$ such that $|s-g|<\e$. Taking a limit as $d\to \infty$ and $\delta\to 0$, this shows that $s(F,n)\leq 1-\mu(D')$. Taking an infimum over $n\in \IN$ and $\e\to 0$, we obtain $s(F)\geq 1-\mu(D)$. Conversely, 
for every  embedding $\p : \bs F_\pm ^n \to \llbracket d_Gd\rrbracket$, we define a map $G\to \llbracket d_Gd\rrbracket$ mapping $s\in \bs F_\pm^n$ to $\p(g)$ where $g$ is any element of $G$ such that $|s-g|<\e$. It is easily seen that this implies that $s(F,n)\leq 1-\mu(D')$, and thus, taking an infimum over $n\in \IN$ and $\e\to 0$, we obtain $s(F)\leq 1-\mu(D)$.  Thus, $s(F)= 1-\mu(D)$.  
\end{proof}

\section{Asymptotic freeness and measure concentration}

We now prove several useful lemmas related to asymptotic freeness. They are based on standard techniques, including the measure concentration property of symmetric groups. We refer to \cite[Theorem 3.9]{Voiculescu1991}, \cite[Theorem 2.7]{Voiculescu1998}, and more recently \cite[Section 3]{Brown2006e}, for analogous results in the context of free entropy. 

Let  $G$ be a finite principal inverse semigroup with fixed tracial state, as in the previous section. Given $d\in\IN$ and a unital trace-preserving embedding $j_d : G \to   \llbracket d\rrbracket$, we write $[d]_G$ for the set of permutations in $[d]\subset \llbracket d\rrbracket$ that commute with $j_d(G)$. We endow $[d]_G$ with the Hamming distance $|s-t|_d:=|\{i=1,\ldots, d\, |\, s(i)\neq t(i)\}|/d$ and  the uniform probability measure $\IP_{G,d}$ assigning $1/{|[d]_G|}$ to every permutation.  

The following is a straightforward generalization of the fact that the symmetric groups $[d]$ endowed with the Hamming distance and the uniform probability measure (which corresponds to $G=\{e\}$) form a L\'evy family (see \cite{Maurey1979} and \cite[Section 3.6]{Gromov1983}).

\begin{lemma}\label{Levy}
Fix integers $d_1<d_2<d_3<\ldots$  and unital trace-preserving embeddings $j_{d_k} : G\to \llbracket d_k\rrbracket$. Then $([d_k]_G,|\cdot|_{d_k},\IP_{G,d_k})_k$ is a L\'evy family.
\end{lemma}

Let $G$ be a finite principal inverse semigroup with fixed tracial state.
Given a pmp equivalence relation $R$ on $(X,\mu)$ and a unital trace-preserving embedding $G\subset  \llbracket R\rrbracket$, we denote by $\Res_{G}$ the trace-preserving restriction map
\begin{align*}
\Res_{G}: \llbracket R\rrbracket &\to \llbracket G\rrbracket\\
s&\mapsto sp
\end{align*}
where $p\in \llbracket R\rrbracket$ is the maximal projection (possibly 0) such that  $sp\in \llbracket G\rrbracket$.

 \begin{lemma}\label{random conjugacy}
Let $G$ be a finite principal inverse semigroup with fixed tracial state and let $n$ be a positive integer. There exists a constant $C_{G,n}$ such that for any  $\e>0$, integers $d_1<d_2<d_3<\cdots$, unital trace-preserving embeddings $j_{d_k} : G\to \llbracket d_k\rrbracket$, and partial transformations $\p_1,\ldots,\p_n,\psi_1,\ldots,\psi_n \in \llbracket d_k\rrbracket$, we have 
\begin{align*}
\int_{[d_k]_G} |\Res_{j_{d_k}(G)}&(\p_1\theta\psi_1\theta^{-1}\cdots \p_n\theta\psi_n\theta^{-1})|\, {\rm d}\IP_{G,d_k}(\theta)\\
&<C_{G,n}\max_i (|\Res_{j_{d_k}(G)}(\p_i)|,|\Res_{j_{d_k}(G)}(\psi_i)|)+\e
\end{align*}
\end{lemma}

This lemma is a generalization of  \cite[Theorem 4.1]{Nica1993} and \cite[Theorem 2.1]{Collins2010} and can be proved by using similar techniques.

\begin{lemma}\label{concentration2}
Let $G$ be a finite principal inverse semigroup with fixed tracial state and let $n$ be a positive integer. There exists a constant $C_{G,n}$ such that for any  $\e>0$, any integers $d_1<d_2<d_3<\cdots$, any unital trace-preserving embeddings $j_{d_k} : G\to \llbracket d_k\rrbracket$, and any $\p_1,\ldots,\p_n,\psi_1,\ldots,\psi_n \in \llbracket d_k\rrbracket$, the sets
\begin{align*}
\Omega_{d_k,\e}(\p_i,\psi_i)&:=\{\theta\in [d_k]_G\mid\\
&\hspace{1.5cm}|\Res_{j_{d_k}(G)}(\p_1\theta\psi_1\theta^{-1}\cdots \p_n\theta\psi_n\theta^{-1})|\\
&\hspace{3cm}<C_{G,n}\max_i (| \Res_{j_{d_k}(G)}(\p_i)|,\,|\Res_{j_{d_k}(G)}(\psi_i)|)+\e\} 
\end{align*}
satisfy $\IP_{G,d_k}(\Omega_{d_k,\e}(\p_i,\psi_i))\to 1$ as $k\to \infty$.
\end{lemma}

\begin{proof}
Let $\eta>0$ and denote by $V_\eta(\Omega_{d_k,\e})$ the $\eta$-neighborhood of $\Omega_{d_k,\e}$ for the Hamming distance. If $\theta' \in V_\eta(\Omega_{d_k,\e})$, then 
\begin{align*}
|\p_1\theta'\psi_1\theta'^{-1}\ldots \p_n\theta'\psi_n\theta^{-1}- \p_1\theta\psi_1\theta^{-1}\cdots \p_n\theta\psi_n\theta^{-1}|&\\
 &\hspace{-4cm}<\sum_{i=1}^n |\theta'\psi_i\theta'^{-1}-\theta\psi_i\theta^{-1}|<2n\eta.
\end{align*}
Thus if $\e>0$ and $\eta<\e/4n$, then $V_\eta(\Omega_{d_k,\e/2})\subset \Omega_{d_k,\e}$.  Since
\[
|\Res_{j_{d_k}(G)}(\p_1\theta\psi_1\theta^{-1}\cdots \p_n\theta\psi_n\theta^{-1})|\leq 1,
\]
Lemma \ref{random conjugacy} shows that
\[
\liminf_{k\to \infty} \IP_{G,d_k}(\Omega_{d_k,\e/2})>0.
\]
By the L\'evy property, $\liminf_{k\to \infty} \IP_{G,d_k}(V_\eta(\Omega_{d_k,\e}))=1$ and thus 
\[
\lim_{k\to \infty} \IP_{G,d_k}(V_\eta(\Omega_{d_k,\e}))=1,
\]
proving the lemma.\end{proof}

\section{The free product formula}

Let $R$ be a pmp equivalence relation on $(X,\mu)$.

\begin{theorem}\label{T- amalgamated} Assume that  $R=R_1*_{R_3}R_2$ is a free product of equivalence relations amalgamated over an amenable subrelation $R_3$. Assume that $R_1$ and $R_2$ are dynamically finitely generated. Then we have
\[
 s_\omega (R)=  s_\omega(R_1)+ s_\omega(R_2)- 1+\mu(D).
\]
where $D$ is a fundamental domain of the finite component of $R_3$.
If furthermore $R_1$ and $R_2$ are  $s$-regular, then $R$ is $s$-regular and
\[
 s (R)=  s(R_1)+ s(R_2)- 1+\mu(D).
\]
\end{theorem}

The proof follows directly from Lemma \ref{L- slower amalgam leq} and Lemma \ref{L - S sup} below.

\begin{lemma}\label{L- slower amalgam leq} Assume that  $R=R_1*_{R_3}R_2$ is a free product of equivalence relations amalgamated over an amenable subrelation $R_3$. Assume that $R_1$ and $R_2$ dynamically finitely generating. Then 
\[
\slower (R)\geq   \slower(R_1)+ \slower(R_2)- 1+\mu(D).
\]
Furthermore, the same inequality holds for $s_\omega$ instead of $\slower$.
\end{lemma}

\begin{proof} Let $F_1\subset \llbracket R_1\rrbracket$ and $F_2\subset \llbracket R_2\rrbracket$ be disjoint subsets. For convenience, we assume that $F_1$ and $F_2$ are symmetric.  
Choose a finite partition $\sP$  fine enough so that any  $s\in \bs F^n$ can be written as a sum $s=\sum_{p_i^s,q_i^s} p_i^s s_iq_i^s$ of $m_s$ pairwise orthogonal elements, where $s_i\in F^n$ and where the projections $p_i^s,q_i^s\in \sP$ are uniquely determined (up to a permutation of indices) and are respectively dominated by the range and the source of $s_i$. Up to refining $\sP$, we may also assume that either $\Res_{\llbracket R_3\rrbracket} (p_i^s s_iq_i^s)=0$ or $p_i^s s_iq_i^s\in \llbracket R_3\rrbracket$ for all $s\in \bs F^n$ and all indices $i=1\ldots m_s$. 
Since $s_i\in F^n$, it can be written as a  product $s_i=s_{i1}\cdots s_{i\ell_i^s}$ where we assume that the length $1\leq\ell_i^s\leq n$ is minimal among all such decompositions and the $s_i$ alternate membership in $F_1^{n}$ and $F_2^n$.  Let $W_i^s$ be the corresponding set of all possible tuples  $(s_{i1}, \ldots, s_{i\ell_i})$.
By our amalgamated free product assumption, for any two such decompositions $\os=(s_{i1}, \ldots, s_{i\ell_i})$  and $\os'=(s_{i1}', \ldots, s_{i\ell_i}')$ in $W_i^s$, there exist $k_{i1}^{\os,\os'},\ldots,k_{i,\ell_i+1}^{\os,\os'}\in\llbracket R_3\rrbracket$ such that $s_{ij}^*s_{ij}\geq k_{i,{j+1}}^{\os,\os'}(k_{i,{j+1}}^{\os,\os'})^*$, $(k_{ij}^{\os,\os'})^*k_{ij}^{\os,\os'}\leq s_{ij}'(s_{ij}')^*$, and
\[
s_{ij}k_{i,{j+1}}^{\os,\os'}=k_{ij}^{\os,\os'}s_{ij}'
\] 
for all $i$ and $j=1\ldots \ell_i$, where $k_{i1}^{\os,\os'}=p_i^s$ and $k_{i,{\ell_i^s+1}}^{\os,\os'}=q_i^s$. In addition,  we may assume up to further refining the partition $\sP$ that either $\Res_{\llbracket R_3\rrbracket} (s_{ij}k_{i,{j+1}}^{\os,\os})=0$ or $s_{ij}k_{i,{j+1}}^{\os,\os}\in \bb{R_3}$, for every index $j$ and every index $i$ such that $\Res_{\llbracket R_3\rrbracket} (p_i^s s_iq_i^s)=0$, and every $\os\in W_i^s$.  Let 
\begin{align*}
&K:= \bigcup_{s\in \bs F_{\pm}^n}  \bigcup_{\forall i, \forall\os,\os'\in W_i^s} \{(k_{ij}^{\os,\os'})^{\pm1}\mid i=1\ldots m_s, j=1\ldots \ell_{i}^s+1\}\\
&\hspace{5cm}\cup\{s_{ij} \mid s_{ij}\in \bb{R_3}\}
\end{align*}
Clearly, $K$ is a finite subset of  $\llbracket R_3\rrbracket$.

Let $\delta>0$ and
let $D$ be a fundamental domain of the finite component of $R_3$. Let $\e>0$ be smaller than $\delta/(200|\sP|\max_{s\in \bs F^n}\max_k (\ell_k^s+10))$.  A direct application of  the Connes--Feldman-Weiss theorem \cite{Connes1981a} shows that there exists a finite principal inverse semigroup  $G\subset \llbracket R_3\rrbracket$, as in Section \ref{S-amenable}, with a fundamental domain $D'$ such that $\mu(D\setminus D')<\e$ and, for any $k\in K$, there exists a $g\in \mathbf{\Sigma} G$ for which $g\leq k$ and $|k-g|< \e$. We may assume that $G=\bs G$ and that $\sP\subset G$. Set $F'_i=F_i\cup G$.

Let $C$ be the constant  (depending only on $G$, $F_i$, and $n$) given by Lemma \ref{concentration2}, and $m_G,n_G$ and $d_G$ be the integers (depending only on $G$, $F_i$, and $n$) as defined in Lemma \ref{G equiv}. Let $\kappa>0$. Assume that  $F_1$ is a dynamical generating set. so we can find an integer $n_0$ such that for all $g\in G$, there exists an $s\in  \mathbf{\Sigma}{F_1}^{n_0}$ such that $|g-s|<\kappa/8$.

Let $\delta'>0$ be smaller than $\delta/200|\sP|C\max_{s\in \bs F^n}\max_k (\ell_k^s+10))$ and $\kappa/4m_G(80+500(2|F_1|+1)^{2n_0})$,  and let $n'$ be greater than both $4n_0$ and $16(n_G+n)^2$. Choose microstates
   $\p_1\in \SA_G(F_1',n',\delta',d_Gd)$ and $\p_2\in \SA_G(F_2',n',\delta',d_Gd)$. 
By Lemma  \ref{G equiv} we have,  for all $d>m_Gn\delta'^{-1}$,
\begin{align*}
N_{\kappa/2}(\SA_G (L,n' ,\delta' ,d_Gd))&\geq N_{m_G\delta'+\kappa/2}(\SA(L,n',\delta'/m_G,d_Gd))\\
&\geq N_{\kappa}(\SA(L,n',\delta'/m_G,d_Gd))
\end{align*}
in either one of the two cases  $L=F_1'\subset  \llbracket R_1\rrbracket$ or $L=F_2'\subset  \llbracket R_2\rrbracket$.   Since the restrictions of $\p_1$ and $ \p_2$ to $G$ are trace-preserving embeddings, we have $|\p_1(D')|=|\p_2(D')|$, and any bijection $\rho : \p_1(D')\to \p_2(D')$ extends by $G$-equivariance to a permutation (still denoted) $\rho\in [ d_Gd ]$ implementing an isomorphism between the inverse semigroups $\p_1(G)$ and $\p_2(G)$. We let $\p_1'(s)=\p_1(s)$ and $\p_2'(s)=\rho^{-1}\p_2(s)\rho$ for $s\in \llbracket R_2\rrbracket$, so that $\p_1'$ and $\p_2'$ coincide on $G$.

Let $[ d_Gd ]_G$ be the subgroup of bijections in $[ d_Gd]$ which commute with $\p_1(G)=\p_2'(G)$. 
Given $\theta\in [ d_Gd ]_G$, we construct a map $\p_\theta$ on $\bs F^{2n}$ as follows.  For each $s\in\bs  F^{2n}$, fix a decomposition of $s_i$ as a product of elements in $(s_{i1}, \ldots, s_{i\ell_i}) \in W_i^s$ where $s=\sum_{p_i^s,q_i^s} p_i^s s_iq_i^s$ and the projections $p_i^s,q_i^s\in \sP$ are uniquely determined up to a permutation of indices. We then define
\[
\p_\theta(s):=\sum_{p_i^s,q_i^s}\p_1(p_i^s)\prod_{j=1}^{\ell_i}\tilde\p_{\alpha_{ij}^s}(s_{ij})\p_1(q_i^s)
\]
where $\tilde \p_{1}:= \p_1$, $\tilde \p_{2}:=\theta\p_2'\theta^{-1}$, and $\alpha_{ij}^s=1$ if $s_{ij}\in F_1^{n}$ and $\alpha_{ij}^s=2$ if $s_{ij}\in F_2^{n}$. This defines $\p_\theta$ for all $s\in\bs F^{2n}$, and we extend $\p_\theta$ to $\llbracket R\rrbracket$  arbitrarily. 

Let us show that $\p_\theta\in \SA(F,n,\delta,d_G d)$ for sufficiently many $\theta$. We divide the proof into two claims.

\begin{claim}
$\p_\theta$ is $(F,n,\delta)$-multiplicative.
\end{claim}
\begin{proof}[Proof of Claim 1] 

We first show that for all $v\in \Sigma F^n$, all indices $i=1,\ldots,m=m_v$ and all $(v_{i1}',\ldots, v_{i\ell_i^v}')\in W^i_v$ we have
\[
\left|\p_1(p_i^v)\prod_{j=1}^{\ell_i^v}\tilde\p_{\alpha_{ij}^v}(v_{ij})\p_1(q_i^v)-\p_1(p_i^v)\prod_{j=1}^{\ell_i^v}\tilde\p_{\alpha_{ij}^v}(v_{ij}')\p_1(q_i^v)\right| <(\ell_i+1)(2\e+52\delta')
\] 
where  $(v_1,\ldots, v_m)\in W_i^v$ is the decomposition chosen in the definition of $\p_\theta$.

Let $k_{ij}\in K$ be such that $v_{ij}'k_{i,j+1}=k_{ij}v_{ij}$, and choose $g_{ij}\in \mathbf {\Sigma}G$ such that  $|g_{ij}-k_{ij}|<\e$, $g_{i1}=p_i^s$ and $g_{i,{\ell_i^s+1}}=q_i^s$.  In particular, $|v_{ij}'g_{i,j
+1}-g_{ij}v_{ij}|<2\e$ and thus, since $n'/16>(n_G+n)$ we get using Lemma \ref{L - sofic approx distance} that
\begin{align*}
|\tilde\p_{\alpha_{ij}^v}(v_{ij})\tilde\p_{\alpha_{ij}^v}(g_{i,j+1})-\tilde\p_{\alpha_{ij}^v}(g_{ij})\tilde\p_{\alpha_{ij}^v}(v_{ij}')|&\\
&\hspace{-2cm}<|\tilde\p_{\alpha_{ij}^v}(v_{ij}g_{i,j+1})-\tilde\p_{\alpha_{ij}^v}(g_{ij}v_{ij}')|+4\delta'\\
&<|v_i'g_{i,j+1}-g_{ij}v_i|+50\delta'<2\e+50\delta'.
\end{align*}
and thus
\begin{align*}&\hspace{0cm}\bigg|\p_1(p_i^v)\tilde \p_{\alpha_{i1}^v}(v_{i1})\cdots  \tilde \p_{\alpha_{i\ell_i^v}^v}(v_{i\ell_i^v}) \p_1(q_i^v)-\p_1(p_i^v)\tilde \p_{\alpha_{i1}^v}(v_{i1}')\cdots  \p_{\alpha_{i\ell_i^v}^v}(v_{i\ell_i^v}') \p_1(q_i^v)\bigg|\\
&<\bigg|\p_1(p_i^v)\tilde \p_{\alpha_{i1}^v}(v_{i1})\cdots  \tilde \p_{\alpha_{i\ell_i^v}^v}(v_{i\ell_i^v})\tilde \p_{\alpha_{i\ell_i^v}^v}(q_i^v) \p_1(q_i^v)\\
&\hspace{3cm}-\p_1(p_i^v)\tilde \p_{\alpha_{i1}^v}(v_{i1}')\cdots  \p_{\alpha_{i\ell_i^v}^v}(v_{i\ell_i^v}') \p_1(q_i^v)\bigg|+\delta'\\
&<\bigg|\p_1(p_i^v)\tilde \p_{\alpha_{i1}^v}(v_{i1})\cdots  \tilde \p_{\alpha_{i\ell_i^v}^v}(g_{i,{\ell_m}})\p_{\alpha_{i\ell_i^v}^v}(v_{i\ell_i^v}') \p_1(q_i^v)\\
&\hspace{3.5cm}-\p_1(p_i^v)\tilde \p_{\alpha_{i1}^v}(v_{i1}')\cdots  \p_{\alpha_{i\ell_i^v}^v}(v_{i\ell_i^v}') \p_1(q_i^v)\bigg|+2\e+51\delta'\\
&\hspace{.2cm}\vdots\\
&<\bigg|\p_1(p_i^v)\tilde \p_{\alpha_{i\ell_i^v}^v}(g_{i1}) \tilde \p_{\alpha_{i1}^v}(v_{i1}')\cdots  \p_{\alpha_{i\ell_i^v}^v}(v_{i\ell_i^v}') \p_1(q_i^v)\\
&\hspace{3cm}-\p_1(p_i^v)\tilde \p_{\alpha_{i1}^v}(v_{i1}')\cdots  \p_{\alpha_{i\ell_i^v}^v}(v_{i\ell_i^v}') \p_1(q_i^v)\bigg|+(\ell_i+1)(2\e+51\delta')\\
&<(\ell_i+1)(2\e+52\delta')
\end{align*}
as claimed.

Let now $s,t\in \bs F^n$ be such that $st\in \bs F^n$ and consider the respective decompositions  $(s_{i1}, \ldots, s_{i\ell_{i,s}})\in W_s^i$, $(t_{j1}, \ldots, t_{j\ell_{j,t}})\in W_t^j$ and $(u_{k1}, \ldots, u_{k\ell_{k,u}})\in W_{u}^k$ of $s$, $t$ and $u=st$, as chosen in the definition of $\p_\theta$. Note that the subset $\{p^u_k\}_k$ of $\sP$ is included in $\{p_i^s\}_i$ and that the subset $\{q^u_k\}_k$ is included in $\{q_j^t\}_j$. In addition, since $u=st$  there exist for any $k$  indices $i_k$ and $j_k$ such that  $q_{i_k}^s=p_{j_k}^t$ and
\[
p_k^u\left (\prod_l u_{kl} \right )q_k^u=p_{i_k}^s\left (\prod_m s_{i_km} \prod_n  t_{j_kn}\right ) q_{j_k}^t.
\]
 In particular, 
 \[
 (s_{i_k1}, \ldots, s_{i_k\ell_{i_k}^s},t_{j_k1},\ldots, t_{k\ell_{j_k}^t})\in W_{u}^k
 \] 
 if $\alpha_{i_k\ell_{i_k}^s}^s\neq \alpha_{j_k1}^t$ and 
 \[
 (s_{i_k1}, \ldots, s_{i_k\ell_{i_k}^s-1}, s_{i_k\ell_{i_k}^s}t_{j_k1},t_{j_k2}, \ldots, t_{k\ell_{j_k}^t})\in W_{u}^k\]
 otherwise. The above computation shows that, in both cases,
 \begin{align*}
&\bigg|\p_1(p_k^u)\prod_{j=1}^{\ell_k^u}\tilde\p_{\alpha_{kj}^u}(u_{kj})\p_1(q_k^u)-\\
&\hspace{2cm}\p_1(p_{i_k}^s)\prod_{m=1}^{\ell_{i_k}^s}\tilde\p_{\alpha_{i_km}^s}(s_{i_km}')\prod_{n=1}^{\ell_{j_k}^t}\tilde\p_{\alpha_{j_kn}^t}(s_{j_kn}')\p_1(q_{j_k}^t)\bigg| <(\ell_k^u+1)(2\e+52\delta')
\end{align*}
for all $k$. Therefore
\begin{align*}
\bigg|\p_1(p_k^u)\prod_{j=1}^{\ell_k^u}\tilde\p_{\alpha_{kj}^u}(u_{kj})\p_1(q_k^u)\\
&\hspace{-3cm}-\p_1(p_{i_k}^s)\prod_{m=1}^{\ell_{i_k}^s}\tilde\p_{\alpha_{i_km}^s}(s_{i_km}')\p_1(q_{i_k}^s)\p_1(p_{j_k}^t)\prod_{n=1}^{\ell_{j_k}^t}\tilde\p_{\alpha_{j_kn}^t}(s_{j_kn}')\p_1(q_{j_k}^t)\bigg|\\
&\hspace{3cm} <(\ell_k^u+1)(2\e+54\delta')
\end{align*}
Summing up over $k$ we obtain
 \[
|  \p_\theta(st)-\p_\theta(s)\p_\theta(t)|< |\sP|\max_k (\ell_k^u+1)(2\e+54\delta')<\delta
\]
which shows that $\p_\theta$ is $(F,n,\delta)$-multiplicative.
\end{proof}

\begin{claim}
For any large enough  $d$, the set of  $\theta\in [d_Gd]_{G}$ such that  $\p_\theta$ is $(F,n,\delta)$-trace-preserving has cardinality at least $\frac1 2|[d_Gd ]_G|$. 
 \end{claim}

\begin{proof}[Proof of Claim 2] 
Fix $s\in F^{n}$. If $i$ is an integer so that $p_i^ss_iq_i^s\in \bb {R_3}$, then for every $j$ there exists a $g_{ij}\in G$ such that  $|s_{ij}-g_{ij}|<\e$ and therefore since both $\tilde \p_1$ and $\tilde\p_2$ are trace preserving isomorphisms on $G$, and since $\theta$ commutes to $\p_1(G)$, we obtain 
\begin{align*}
|\p_\theta(p_i^ss_iq_i^s) &-\p_1(p_i^sg_{i1}\cdots g_{i\ell_i^s}q_i^s)|\\
&= |\p_1(p_i^s)\tilde\p_{\alpha_{i1}^s}(s_{i1})\cdots \p_{\alpha_{i\ell_s^i}^s}(s_{i\ell_s^i})\p_1(q_i^s) -\p_1(p_i^sg_{i1}\cdots g_{i\ell_i^s+1}q_i^s)|\\
&\leq |\p_1(p_i^s)\p_1(g_{i1})\cdots \p_1(g_{i\ell_s^i})\p_1(q_i^s) -\p_1(p_i^sg_{i1}\cdots g_{i\ell_i^s+1}q_i^s)|+\ell_i^s\e\\
&\leq (\ell_i^s+2)(\e+\delta')
\end{align*}
where the last inequality follows from Lemma \ref{L - sofic approx distance}. Thus in this case
\begin{align*}
|\tr\circ\p_\theta(p_i^ss_iq_i^s) -\tau(p_i^ss_iq_i^s)|&\leq 5(\ell_i^s+2)(\e+\delta').
\end{align*}
We now assume that $\Res_{\bb {R_3}}(p_i^ss_iq_i^s)=0$. Since
$s_{ij}k_{i,{j+1}}^{\os,\os}=k_{ij}^{\os,\os}s_{ij}$, 
denoting $\os=(s_{i1}, \ldots, s_{i\ell_i})$, we have 
\begin{align*}
\bigg| \p_{1}( p_i^s)\prod_j\tilde\p_{\alpha_{ij}^s}(s_{ij})\p_{1}(q_i^s)- \p_{1}( p_i^s)\prod_j\tilde\p_{\alpha_{ij}^s}(s_{ij}k_{i,{j+1}}^{\os,\os})\p_{1}(q_i^s)\bigg|&\leq2\delta'\ell_i^s
\end{align*}
Let us first assume that $\Res_{\bb {R_3}}(s_{ij}k_{i,{j+1}}^{\os,\os})=0$ for all $j$.
Then for any $d$ large enough and for at least half of the $\theta\in [d_Gd ]_G$ we have by  Lemma \ref{concentration2}  applied to to $[d_Gd ]_G$
\begin{align*}
\bigg| \Res_{\p_1(G)} \bigg(\p_{1}( p_i^s)\prod_j\tilde\p_{\alpha_{ij}^s}(s_{ij})\p_{1}(q_i^s)\bigg)\bigg|&\leq\bigg| \Res_{\p_1(G)}\bigg (\prod_j\tilde\p_{\alpha_{ij}^s}(s_{ij}k_{i,{j+1}}^{\os,\os})\bigg)\bigg|+2\delta'\ell_i^s\\
&<C\max_j(|\Res_{\p_1(G)} \p_{\alpha_{ij}^s}'(s_{ij}k_{i,{j+1}}^{\os,\os})|)+\e +2\delta'\ell_i^s
\end{align*}
Let  $g\in  G$ and let $f\leq \p_1(g)\p_1(g)^{-1}$, $f\leq  \p_{\alpha_{ij}}'(s_{ij}k_{i,{j+1}}^{\os,\os}) \p_{\alpha_{ij}}'(s_{ij}k_{i,{j+1}}^{\os,\os})^{-1}$ be such that 
\[
f \p_1(g) =\Res_{\p_1(G)} \tilde  \p_{\alpha_{ij}}'(s_{ij}k_{i,{j+1}}^{\os,\os})=f \p_{\alpha_{ij}}'(s_{ij}k_{i,{j+1}}^{\os,\os}).
\] 
Since $| \p_{\alpha_{ij}}'(s_{ij}k_{i,{j+1}}^{\os,\os}) \p_1(g)^{-1}- \p_{\alpha_{ij}}'(s_{ij}k_{i,{j+1}}^{\os,\os}g^{-1})|<4\delta'$, we see that   
\[
|f-f\p_{\alpha_{ij}}'(s_{ij}k_{i,{j+1}}^{\os,\os}g^{-1})|<4\delta'.
\] On the other hand, since $\Res_{\bb {R_3}}(s_{ij}k_{i,{j+1}}^{\os,\os})=0$, we have  $\tau(s_{ij}k_{i,{j+1}}^{\os,\os}g^{-1})=0$ and thus $|\tr(\p_{\alpha_{ij}}'(s_{ij}k_{i,{j+1}}^{\os,\os}g^{-1}))|< \delta'$. Therefore $|\tr(f)|<5\delta'$, and we deduce that
\[
|\Res_{\p_1(G)} \p_{\alpha_{ij}^s}'(s_{ij}k_{i,{j+1}}^{\os,\os})|<5\delta'.
\] 
In particular, $| \tr (\p_{1}( p_i^s)\prod_j\tilde\p_{\alpha_{ij}^s}(s_{ij})\p_{1}(q_i^s))|<5C\delta'+\e+2\delta'\ell_i^s$.
Thus we obtain

 \begin{align*}
|\tr\circ \p_\theta(s)-\tau(s)|&\leq\sum_i |\tr\circ\p_\theta(p_i^ss_iq_i^s) -\tau(p_i^ss_iq_i^s)|\\
&\leq |\sP|(5C\delta'+\e+10\max_i(\ell_i^s+2)(\e+\delta'))\\
&<\delta.
\end{align*}

The last case is the situation where $s_{ij}k_{i,{j+1}}^{\os,\os}\in \bb {R_3}$ for some of the indices $j$ (where $i$ is being fixed). This can be handled as follows. Say  $j$ is an index such that $s_{ij}k_{i,{j+1}}^{\os,\os}\in \bb {R_3}$ but $\Res_{\bb{R_3}}(s_{i,j-1}k_{i,{j}}^{\os,\os})=0$. Let $g_{ij}\in G$ be such that $|s_{ij}k_{i,{j+1}}^{\os,\os}-g_{ij}|<\e$. Then 
\begin{align*}
&|\p_{\alpha_{i,j-1}^s}'(s_{i,j-1}k_{i,{j}}^{\os,\os})\p_{\alpha_{ij}^s}'(s_{ij}k_{i,{j+1}}^{\os,\os})-\p_{\alpha_{i,j-1}^s}'(s_{i,j-1}k_{i,{j}}^{\os,\os}g_{ij})|\\
&\hspace{1.6cm}\leq |\p_{\alpha_{i,j-1}^s}'(s_{i,j-1}k_{i,{j}}^{\os,\os})\p_{\alpha_{ij}^s}'(g_{ij})-\p_{\alpha_{i,j-1}^s}'(s_{i,j-1}k_{i,{j}}^{\os,\os}g_{ij})|+\e+40\delta'\\
&\hspace{2cm}= |\p_{\alpha_{i,j-1}^s}'(s_{i,j-1}k_{i,{j}}^{\os,\os})\p_{\alpha_{i,j-1}^s}'(g_{ij})-\p_{\alpha_{i,j-1}^s}'(s_{i,j-1}k_{i,{j}}^{\os,\os}g_{ij})|+\e+40\delta'\\
&\hspace{2.6cm}\leq \e+41\delta'.
\end{align*}
However, since $\Res_{\bb{R_3}}(s_{i,j-1}k_{i,{j}}^{\os,\os})=0$ we now have $\Res_{\bb{R_3}}(s_{i,j-1}k_{i,{j}}^{\os,\os}g_{ij})=0$. Iterating this step at most $\ell_i^s$ times brings us back to the situation considered above, up to an error of $\ell_i^s \e+41\delta'$. Then the same proof applies with appropriately adjusted error terms.

Finally, for any large enough  $d$, the set of permutations $\theta\in [d_Gd ]_G$ satisfying $\p_\theta$ is $(F,n,\delta)$-trace-preserving has  has cardinality at least $\frac 1 2 |[d_Gd ]_G|$. 
\end{proof}

The above two claims show that for every $\p_1\in \SA_G(F_1',n',\delta',d_Gd)$, every $\p_2\in \SA_G(F_2',n',\delta',d_Gd)$ 
and at least half of the permutations $\theta\in [ d_Gd ]_G$, we get a map
$\p_\theta\in \SA_G(F,n,\delta,d_G d)$. 
Clearly, if the maps $\p_{\theta}$ and $\psi_{\theta'}$ associated to two triples $(\p_1,\p_2,\theta)$ and  $(\psi_1,\psi_2,\theta')$ coincide on $F$, then $\p_1$, $\psi_1$ coincide on $F_1$ and  $\theta\p_2'\theta^{-1}$ and $\theta'\psi_2'\theta'^{-1}$ coincide on $F_2$. 
In particular,  we can find a permutation $\gamma\in [d_G d]$ such that $\p_2$ and $\gamma\psi_2\gamma^{-1}$ coincide on $F_2$. 
Furthermore, for every $g\in G$, we can find an element $v\in \mathbf{\Sigma} {F_1}^{n_0}$ such that $|g-v|<\kappa/8$. Suppose that $v\in {F_1}^{n_0}$. Since $\p_1=\psi_1$ on $F_1$, we have, writing $v=\prod_{i=1}^{n_0} v_i$ as a product of elements of ${F_1}$, that
\begin{align*}
 |\p_1(v)-\psi_1(v)|&\leq 2n_0\delta'.
\end{align*}
If $v=\sum_{i=1}^k v_i\in  \mathbf{\Sigma}F_1^{n_0}$, then by Lemma \ref{L -Linearity} shows that
\[
\bigg |\p_1(v) - \sum_{i=1}^k\p_1(v_i)\bigg|< 150(2|F_1|+1)^{2n_0} \delta'.
\]
and similarly
\[
\bigg |\psi_1(v) - \sum_{i=1}^k\psi_1(v_i)\bigg|< 150(2|F_1|+1)^{2n_0} \delta'.
\]
whence
\begin{align*}
| \p_1 (v) - \psi_1 (v) |&\leq \bigg|\sum_{i=1}^k\p_1(v_i)\pi_i(\p_1(v_1),\ldots,\p_1(v_k))-\sum_{i=1}^k\psi_1(v_i)\pi_i(\psi_1(v_1),\ldots,\psi_1(v_k))\bigg|\\
&\hspace{6cm}+ 300(2|F_1|+1)^{2n_0}\delta'\\ 
&\leq \sum_{i=1}^k|\p_1(v_i)\pi_i(\p_1(v_1),\ldots,\p_1(v_k))-\psi_1(v_i)\pi_i(\psi_1(v_1),\ldots,\psi_1(v_k))|\\
&\hspace{6cm}+ 300(2|F_1|+1)^{2n_0}\delta'\\ 
&\leq \sum_{i=1}^k|\p_1(v_i)-\psi_1(v_i)|+ 80k(k-1)\delta'+ 300(2|F_1|+1)^{2n_0}\delta'\\ 
&\leq (220n_0k\delta'+ 80k(k-1)\delta'+ 300(2|F_1|+1)^{2n_0}\delta'\\
&\leq 500(2|F_1|+1)^{2n_0}\delta',
\end{align*}
and thus
\begin{align*}
 |\p_1(g)-\psi_1(g)|&\leq |\p_1(g)-\p_1(v)|+|\p_1(v)-\psi_1(v)|+|\psi_1(v)-\psi_1(g)|\\
 &< 2|g-v|+ 80\delta'+500(2|F_1|+1)^{2n_0}\delta'\\
 &<\kappa/4+ 80\delta'+500(2|F_1|+1)^{2n_0}\delta'<\kappa/2.
\end{align*}
Therefore, $|\p_1-\psi_1|_{F_1'}<\kappa/2$. On the other hand, since $\p_1$ and $\p_2'$ coincide on $G$, and $\psi_1$ and $\psi_2'$ coincide on $G$,  we also obtain that  $|\theta\p_2'(g)\theta^{-1}-\theta'\psi_2'(g){\theta'}^{-1}|<\kappa/2$ for every $g\in G$. Therefore,
$|\p_2-\gamma\psi_2\gamma^{-1}|_{F_2'}<\kappa/2$,
and we deduce that
\begin{align*}
\NSA(F,n,\delta,d_G d)\hspace{-2cm}&\\
&\geq N_{\kappa/2}(\SA_G(F_1',n',\delta',d_Gd))N_{\kappa/2}(\SA_G(F_2',n',\delta',d_Gd))\cdot\frac{|[ d_Gd ]_G|}{2 (d_Gd)!}\\
&\geq N_{\kappa}(\SA(F_1',n',\delta'/m_Gn,d_Gd))N_{\kappa}(\SA(F_2',n',\delta'/m_Gn,d_Gd))\cdot\frac{|[ d_Gd ]_G|}{2 (d_Gd)!}
\end{align*}
As in Corollary \ref{L - s(G) G finite}, we have 
\[
\log |[ d_Gd ]_G|\sim \mu(D') d_Gd\log d_Gd
\] 
as $d\to \infty$, using Stirling approximation.
Since $\mu(D')\geq \mu(D)-\e$ we obtain, by taking the limit infimum over $d$ and an infimum over $\delta'$ and $n'$, that 
\begin{align*}
\underline s(F,n,\delta)&\geq \underline s_{\kappa}(F_1',n',\delta'/m_Gn) +\underline s_{\kappa}(F_2',n',\delta'/m_Gn) + \mu(D') - 1\\
&\geq \underline s_{\kappa}(F_1') +\underline s_{\kappa}(F_2') - 1+\mu(D)-\e.
\end{align*}
Note that the same inequality holds for $s_\omega$ as well, by taking a limit along the ultrafilter $\omega$ instead of a limit infimum. Thus,  taking the limit over $\kappa$ (using Lemma \ref{L-ineq lim sup epsilon})  we obtain  the inequality
\[
\underline s(F,n,\delta)\geq \underline s(F_1') +\underline s(F_2')- 1+\mu(D)-\e.
\]
Since $F_1$ is dynamically generating $R_1$, we have $\underline s(F_1')=\underline s(R_1)$  from Theorem \ref{P-generating}, and thus we have shown that, for any $n\in \IN$, $\delta>0$, and any $\e>0$ sufficiently small, there exists a finite inverse semigroup $G$ inside $\bb{R_3}$ such that 
\[
\slower (F_1\cup F_2,n,\delta)\geq   \slower(R_1)+ \slower(F_2\cup G)- 1+\mu(D)-\e.
\] 
If in addition $F_2$ is  dynamically generating, then we have both $\slower(F_2\cup G)=\slower(R_2)$
 and $\slower(F_1\cup F_2)=\slower(R)$. Therefore, by taking a limit as $\e\to 0$ and an infimum over $n$ and $\delta$, we obtain
\[
\slower (R)\geq   \slower(R_1)+ \slower(R_2)- 1+\mu(D).
\] 
which establishes Lemma \ref{L- slower amalgam leq}.
\end{proof}

\begin{lemma}\label{L - S sup} Assume that  $R=R_1*_{R_3}R_2$ is a free product of equivalence relations amalgamated over an amenable subrelation $R_3$. Assume that $R_1$ and $R_2$ dynamically finitely generated. Then  
\[
 s(R)\leq  s(R_1)+ s(R_2)- 1+\mu(D).
\]
Furthermore, the same inequality holds for $s_\omega$ instead of $s$.
\end{lemma}

\begin{proof} Let $F_1\subset \llbracket R_1\rrbracket$ and $F_2\subset \llbracket R_2\rrbracket$ be finite subsets and let $\e>0$. As in Lemma \ref{L- slower amalgam leq}, the Connes--Feldman-Weiss theorem gives us a finite inverse semigroup $G\subset \llbracket R_3\rrbracket$ of support $X$ with a fundamental domain $D'$  such that $\mu(D')\leq \mu(D)+\e$.
Set $F'_i=F_i\cup G$. By Lemma  \ref{G equiv}, there exists an integer $d_G$ depending only on $G$ such that
\[
 s(L) = \inf_{n\in \mathbb{N} }\inf_{\delta>0} \limsup_{d\to\infty} \frac{1}{d_Gd \log d_Gd} \log \NSA_G (L,n ,\delta ,d_Gd),
\]  
in either one of the cases  $L=F_1'\subset  \llbracket R_1\rrbracket$, $L=F_2'\subset  \llbracket R_2\rrbracket$ or $L=F_1'\cup F_2'\subset  \llbracket R\rrbracket$.

Clearly, every $\p\in \SA_G (F_1'\cup F_2',n ,\delta ,d_Gd)$ gives by restriction to $\llbracket R_1\rrbracket\subset \llbracket R\rrbracket$ and $\llbracket R_2\rrbracket\subset \llbracket R\rrbracket$ two maps $\varphi_1\in \SA_G (F_1',n ,\delta ,d_Gd)$ and $\varphi_2\in \SA_G (F_2',n ,\delta ,d_Gd)$ respectively.
Consider the map $\Theta$ defined by
\begin{align*}
\Theta: \SA_G(F_1'\cup F_2',n ,\delta ,d)\times [d_Gd]&\to \SA_G(F_2',n ,\delta ,d)\times \SA_G(F_2',n ,\delta ,d)\\
(\p,\theta)&\mapsto (\p_1,\theta\p_2\theta^{-1})
\end{align*}
If $(\p,\theta)$ and $(\p',\theta')$ have the same image under $\Theta$, then $\p_1=\p_1'$ on $F_1'$ and  $\theta\p_2\theta^{-1}=\theta'\p_2'\theta'^{-1}$ on $F_2'$. Let $\rho=\theta'^{-1}\theta$. The first condition implies that $\p_{|G}=\p'_{|G}$ and the second that $\rho\p_{|G}\rho^{-1}=\p'_{|G}=\p_{|G}$. In particular, $\rho$ belongs to the commutant $[d_Gd]_G$ of $\p(G)$ in $[d_Gd]$. It follows that 
\[
 \NSA_G(F_1'\cup F_2',n ,\delta ,d_Gd)\leq  \NSA_G(F_1',n ,\delta ,d_Gd) \NSA_G(F_2',n ,\delta ,d_Gd)\cdot \frac{|[d_Gd]_G|}{(d_Gd)!}
\]
so that, as in Lemma \ref{L- slower amalgam leq}, we obtain
$s(F_1'\cup F_2')+1-\mu(D')\leq s(F_1')+s(F_2')$,
so 
\[
s(F_1'\cup F_2')\leq s(F_1')+s(F_2')- 1+\mu(D)+\e.
\] 
If $F_1$ and $F_2$ are dynamical generating sets, Theorem \ref{P-generating} applies and therefore
\[
 s(R)\leq  s(R_1)+ s(R_2)- 1+\mu(D).
\]
as announced, by taking a limit as $\e\to 0$.
\end{proof}

\begin{remark}
The above proof shows that in fact the following stronger result holds: assume that  $R=R_1*_{R_3}R_2$ is a free product of finitely generated equivalence relation amalgamated over a common subrelation $R_3$, and let $R_3'\subset R_3$ be an amenable subrelation. Then we have:
\[
 s(R)\leq  s(R_1)+ s(R_2)- 1+\mu(D)
\]
where $D$ is a fundamental domain of the finite component of $R_3'$. In particular if $R_3$ is diffuse, then $ s(R)\leq  s(R_1)+ s(R_2)- 1$.
\end{remark}

In the next corollary, the assertion that treeable equivalence relations are sofic is a recent result of Elek--Lippner \cite{Elek2010d}.  

\begin{corollary}\label{C-treeable}
Let $R$ be an ergodic finitely generated pmp equivalence relation. If $R$ is treeable, it is $s$-regular and $s(R)=\cost(R)$ (in particular, $R$ is sofic).
\end{corollary}

\begin{proof} Since $s(R)\leq \cost(R)$ holds in general,  we have to prove the other inequality.
By Hjorth's theorem \cite{Hjorth2006}, we can find a generating set of $R$ of the form $F=\{s_1,s_2,\ldots s_n\}$ where $s_1,\ldots s_{n-1}\in [R]$ and $s_n\in \llbracket R_1\rrbracket$, such that $s_1$ is ergodic and, if $R_i$ denotes the relation generated by $s_i$, then $R=R_1*\cdots * R_n$  (see the proof of \cite[Theorem 28.3]{Kechris2004}). Each equivalence relation $R_i$ is amenable and therefore, every finite generating set is $s$-regular by Corollary \ref{L - s(G) G finite}.  Since $s_1$ is ergodic, we may replace it in $F$ by a dynamical generating set consisting of two partial isomorphisms of cost $\frac 1 2$. The corresponding finite set $F'$ is then a finite dynamical generating set for $R$, and thus $\slower(R)=\slower(F')$ and $\slower(R)=\slower(F')$.  Proceed by recurrence on $n$ and assume that $\slower(R_1*\cdots * R_{n-1})\geq\cost(R_1*\cdots * R_{n-1})$. Applying  Lemma \ref{L- slower amalgam leq}  we see that for every $n$, every $\delta>0$ and any $\e>0$ sufficiently small, there exists a partition $G$ of $X$ such that $\slower(F',n,\delta)\geq \slower(R_1*\cdots * R_{n-1}) + \slower(\{s_n\}\cup G)-\e\geq \cost(R_1*\cdots * R_{n-1})+\mu(\dom s_n)-\e=\cost(R)-\e$. Taking a limit as $\e\to0$ and an infimum over $n$ and $\delta$, it follows that $\slower(R)\geq\cost(R)$. Therefore, $R$ is $s$-regular and $s(R)=\cost(R)$.
\end{proof}

\bibliographystyle{amsalpha}

\end{document}